\documentclass[a4paper]{amsart}

\usepackage{etex}
\usepackage{aliascnt}
\usepackage{bbm} 
\usepackage{pbox}
\usepackage{booktabs}
\usepackage{arydshln}
\usepackage{pinlabel, caption, subcaption}
\usepackage{esvect}
\usepackage{cals}
\usepackage{verbatim}
\usepackage{stmaryrd}
\DeclareMathOperator{\Mod}{Mod}
\DeclareMathOperator{\Tor}{Tor} 

\DeclareMathOperator{\VB}{VB} 
\DeclareMathOperator{\Ch}{Ch}
 
\DeclareMathOperator{\Ab}{Ab}
\DeclareMathOperator{\CB}{CB}

\DeclareMathOperator{\Top}{Top}
\DeclareMathOperator{\Set}{Set}
\DeclareMathOperator{\Fun}{Fun}
\DeclareMathOperator{\Ind}{Ind}

\newcommand{\m}{\to}

\newcommand{\cC}{\mathcal{C}}

\newcommand{\cU}{\mathcal{U}}

\providecommand{\Link}{\ensuremath\mr{Link}}

\newcommand{\Z}{\mathbb{Z}}
\newcommand{\N}{\mathbb{N}_0}

\newcommand{\R}{\mathbb{R}}

\usepackage{lmodern}
\usepackage{booktabs}
\usepackage[dvipsnames,svgnames,x11names,hyperref]{xcolor}
\usepackage{mathtools,url,graphicx,verbatim,amssymb,enumerate,stmaryrd,nicefrac}
\usepackage[pagebackref,colorlinks,citecolor=blue,linkcolor=blue,urlcolor=blue,filecolor=blue]{hyperref}
\usepackage{microtype}
\usepackage[margin=1.25in]{geometry}
\usepackage{pdflscape}

\usepackage{tikz, tikz-cd}
\usetikzlibrary{matrix,arrows,decorations.pathmorphing}
\usetikzlibrary{arrows,decorations.pathmorphing,backgrounds,fit,positioning,shapes.symbols,chains,calc}

\numberwithin{theoremcounter}{section}
\newaliascnt{theoremauto}{theoremcounter}

\newaliascnt{Defauto}{theoremcounter}

\newaliascnt{exampleauto}{theoremcounter}

\newaliascnt{lemmaauto}{theoremcounter}

\newaliascnt{propositionauto}{theoremcounter}

\newaliascnt{corollaryauto}{theoremcounter}

\newaliascnt{remarkauto}{theoremcounter}

\newaliascnt{notationauto}{theoremcounter}

\newaliascnt{claimauto}{theoremcounter}

\newaliascnt{warningauto}{theoremcounter}

\newaliascnt{questionauto}{theoremcounter}

\newaliascnt{discussionauto}{theoremcounter}

\newaliascnt{computationauto}{theoremcounter}

\newaliascnt{conjectureauto}{theoremcounter}

\newaliascnt{convauto}{theoremcounter}

\newtheorem{theorem}[theoremauto]{Theorem}
\newtheorem{lemma}[lemmaauto]{Lemma}
\newtheorem{proposition}[propositionauto]{Proposition}
\newtheorem{corollary}[corollaryauto]{Corollary}
\newtheorem*{corollary*}{Corollary}
\newtheorem{conjecture}[conjectureauto]{Conjecture}

\theoremstyle{definition}
\newtheorem{definition}[Defauto]{Definition}
\newtheorem{notation}[notationauto]{Notation}

\theoremstyle{remark}
\newtheorem{example}[exampleauto]{Example}
\newtheorem{remark}[remarkauto]{Remark}

\newtheorem{question}[questionauto]{Question}

\newcommand{\mr}[1]{{\rm #1}}

\newcommand{\rank}{\mr{rank}}
\newcommand{\Star}{\mr{Star}}

\DeclareMathOperator*{\colim}{colim}

\DeclareMathOperator{\V}{Vert}

\newcommand{\GL}{\mr{GL}}

\newcommand{\St}{\mr{St}}

\newcommand{\Lk}{\mr{Link}}

\newcommand{\jw}[1]{\marginpar{\tiny\textcolor{violet}{jw: #1}}}

\title[On  rank filtrations of algebraic $K$-theory and Steinberg modules]{On  rank filtrations of algebraic $K$-theory\\ and Steinberg modules}

 \author{Jeremy Miller}\thanks{Jeremy Miller was supported in part by NSF Grant DMS-2202943 and a Simons Foundation Collaboration Grant}
 \email{jeremykmiller@purdue.edu}  
\address{Purdue University \\
Department of Mathematics \\
 	 150 N. University \\
 	 West Lafayette IN, 47907 \\USA}

\author{Peter Patzt}
\email{ppatzt@ou.edu}
\address{University of Oklahoma\\
Department of Mathematics \\
 	 601 Elm Av \\
 	 Norman OK, 73019 \\USA}
	 \thanks{Peter Patzt was supported by NSF grant DMS-2405310, a Simons Foundation Collaboration Grant, the Danish National Research Foundation through the Copenhagen Centre
for Geometry and Topology (DNRF151), and the European Research Council under the European Union’s
Seventh Framework Programme ERC Grant agreement ERC StG 716424 - CASe, PI Karim Adiprasito.}
   
\author{Jennifer C. H. Wilson}
\email{jchw@umich.edu}
\address{University of Michigan \\ Department of Mathematics \\
 	 530 Church St\\
 	 Ann Arbor MI, 48109 \\USA}
\thanks{Jennifer Wilson was supported in part by NSF grant DMS-1906123 and NSF CAREER grant DMS-2142709}

\AtBeginDocument{%
	\def\MR#1{}}

\date{}

\begin{document}
	
\begin{abstract} Motivated by his work on the stable rank filtration of algebraic $K$-theory spectra, Rognes defined a simplicial complex called the common basis complex and conjectured that this complex is highly connected for local rings and Euclidean domains. We prove this conjecture in the case of fields. Our methods give a novel description of this common basis complex of a PID as an iterated bar construction on an equivariant monoid built out of Tits buildings. We also identify the Koszul dual of a certain equivariant ring assembled out of Steinberg modules.

 \end{abstract} 

\maketitle

\tableofcontents


\section{Introduction}

\subsection{The stable rank filtration and Rognes' connectivity conjecture}

Given a PID $R$, let ${\CB}_n(R)$ denote Rognes' common basis complex. That is, ${\CB}_n(R)$ is the simplicial complex with vertices the proper nonzero summands of $R^n$ such that $\{V_0,\dots, V_p\}$ forms a simplex if there is a basis for $R^n$ such that each $V_i$ is a span of a subset of that basis. In this case, we say the set $\{V_0,\dots, V_p\}$ has a \emph{common basis}.   

One of the most popular models for algebraic $K$-theory is Waldhausen's iterated $S_\bullet$-construction \cite{WaldhausenSource}.  Filtering the Waldhausen $S_\bullet$-construction by rank yields a filtration of the algebraic $K$-theory spectrum: 
\[*=\mathcal{F}_0K(R) \subset \mathcal{F}_1K(R) \subset \dots \subset K(R).\]
Rognes  \cite{Rog1} proved that the associated graded of this filtration is the general linear group homotopy orbits of the suspension spectrum of the suspension  of the common basis complex: \[\mathcal{F}_{n}K(F)/\mathcal{F}_{n-1}K(F) \simeq \left(\Sigma^{\infty}\Sigma {\CB}_n(R) \right)_{h\GL_n(R)}.\] This stable rank filtration was a key ingredient in Rognes' proof that $K_4(\Z) \cong 0$ \cite{RognesK4}. Based on calculations for $n=2$ and $3$, Rognes \cite[Conjecture 12.3]{Rog1} made the following connectivity conjecture.

\begin{conjecture}[Rognes' connectivity conjecture]
For $R$ a local ring or Euclidean domain, ${\CB}_n(R)$ is $(2n-4)$-connected.
\end{conjecture}

We resolve this conjecture for $R=F$ a field.

\begin{theorem}\label{ConnectivityThm}
For $F$ a field, $ {\CB}_n(F)$ is $(2n-4)$-connected. 
\end{theorem}

Rognes  \cite{Rog1} proved that  $\widetilde{H}_i({\CB}_n(R))$ vanishes for $i>2n-3$. It follows that the reduced homology of ${\CB}_n(F)$ is concentrated in degree $2n-3$. Rognes named the conjectural single non-vanishing reduced homology group the \emph{stable Steinberg module} $\St^{\infty}_n(R)$ \cite[Definition 11.3]{Rognes96}. \autoref{ConnectivityThm} shows that the homology of the associated graded of the stable rank filtration of the $K$-theory of a field is the homology of $\GL_n(F)$ with coefficients in the stable Steinberg module:
$$H_i\big(\mathcal{F}_n K(F)/ \mathcal{F}_{n-1}K(F)\big) \cong H_{i-2n+2}\big(\GL_n(F); \St^{\infty}_n(F)\big).$$
In particular this result implies a vanishing line on the $E^1$ page of the spectral sequence associated to the stable rank filtration. 

We note that in the case that $F$ is an infinite field, Galatius--Kupers--Randal-Williams recently proved that $(\Sigma^{\infty} \Sigma {\CB}_n(F))_{h\GL_n(F)}$ is $(2n-3)$-connected \cite[Theorem C]{e2cellsIV}. In particular, their results imply the same vanishing line on the $E^1$ page of the spectral sequence associated to the stable rank filtration in the case of infinite fields. Our result implies $(\Sigma^{\infty} \Sigma {\CB}_n(F))_{h\GL_n(F)}$ is $(2n-3)$-connected for all fields.

\subsection{Higher Tits buildings and the proof strategy} 

The rank filtration not only gives a filtration of the $K$-theory spectrum but also gives a filtration of each of the spaces in the Waldhausen model of the K-theory  spectrum. Rognes proved that the associated graded of the rank filtration of the $k$th space is the reduced homotopy orbits of a $\GL_n(R)$ action on a space $D^k_n(R)$.\footnote{Rognes' notation differs from our notation having subscripts and superscripts flipped.} This space can be thought of as a $k$-dimensional version of the Tits building. The definition of $D^k_n(R)$ is somewhat involved (see \autoref{Dab}, following Rognes \cite[Definition 3.9]{Rog1}) so we will only describe a model of its desuspension here. 

The Tits building $T_n(R)$ is the realization of the poset of proper nontrivial summands of $R^n$ ordered by inclusion. A simplex in $T_n(R)$ is a flag of summands, that is, a collection of summands totally ordered under inclusion. Let $T^k_n(R)$ denote the subcomplex of the $k$-fold join \[\underbrace{T_n(R) * \dots * T_n(R)}_{k \text{ times}} \] of simplices that admit a common basis.  There is an equivalence $D_n^k(R) \simeq \Sigma^{k+1} T_n^k(R)$ (\autoref{suspend}). 

The complexes $\{D^k_n(R)\}_n$ assemble to form a kind of equivariant monoid. Let $\GL(R)$ denote the groupoid of general linear groups $\{\GL_n(R)\}_n$ viewed as a symmetric monoidal category with block sum. 
The category $\Fun(\GL(R),{\Top_*}  )$ of functors to based spaces has  the structure of a symmetric monoidal category as well. The monoidal operation is given by Day convolution (\autoref{Day}). 
 Let $D^k(R)$ denote the functor $n \mapsto D^k_n(R)$. Galatius--Kupers--Randal-Williams \cite{e2cellsIV} observed that  $D^k(R)$ has the structure of an augmented graded-commutative monoid object in this category. Concretely, the multiplication map is the data of $(\GL_n(R) \times \GL_m(R))$-equivariant maps: \[D^k_n(R) \wedge D^k_m(R) \m D^k_{n+m}(R). \] It satisfies an equivariant version of the associativity and commutativity axioms. This augmented monoid structure allows us to make sense of bar constructions. Our main technical result is the following. 

\begin{lemma} \label{mainLemma} For $R$ a PID and $k \geq 1$, there is an equivalence $B D^k(R) \simeq D^{k+1}(R)$.
\end{lemma}

Since bar constructions increase connectivity by one (see \autoref{Bconn}) and, by \autoref{colimTD},  \[{\CB}_n(R) \simeq \colim_{k \m \infty} T_n^k(R),\] we obtain the following corollary. 

\begin{corollary} Let $R$ be a PID. If $T_n^2(R)$ is $(2n-4)$-connected, so is ${\CB}_n(R)$. \end{corollary}

When $F$ is a field,  $T_n^2(F)$ is simply the join $T_n(F) * T_n(F)$, since every pair of flags has a common basis. This is equivalent to the classical property that every pair of chambers in the Tits building of $F^n$ are contained in a common apartment.  The Tits building $T_n(R)$ is known to be $(n-3)$-connected for $R$ a PID by the Solomon--Tits Theorem (here, \autoref{SolomonTits}), and  \autoref{ConnectivityThm} follows. The identification $T_n^2(F) = T_n(F) * T_n(F)$ is the one point in the paper that we must specialize to fields instead of general PIDs. For a general PID this equality fails, for example,
$$\left( \Z \begin{bmatrix} 1\\1 \end{bmatrix} \right) *  \left( \Z \begin{bmatrix} 1\\-1 \end{bmatrix} \right)$$ is a simplex of $T_2(\Z) * T_2(\Z)$ but not $T_2^2(\Z)$; see \autoref{ExampleIncompatibleLines}.

\subsection{Steinberg modules} 

The Steinberg module is defined to be \[\St_n(R):=\widetilde H_{n-2}(T_n(R)) \cong H_n(D_n^1(R)),\] for $n>0$, and by convention $\St_0(R)=\Z$. This is an important object in representation theory, algebraic $K$-theory, and the cohomology of arithmetic groups. The assignment $n \mapsto \St_n(R)$ assembles to a functor $\St(R)$ from $\GL(R)$ to abelian groups. Miller--Nagpal--Patzt \cite{MNP} described a multiplication map \[\St_n(R) \otimes \St_m(R) \m \St_{n+m}(R) \] involving  ``apartment concatenation'' which makes $\St(R)$ into an connected augmented graded-commutative monoid object in $\Fun(\GL(R),\Ab)$, the category of functors to abelian groups. Using this structure, we can make sense of the groups $\Tor^{\St(F)}_i(\Z,\Z)_n$. Here $i$ denotes the homological grading and $n$ denotes the grading associated to the groupoid \[\GL(R)=\bigsqcup_n \GL_n(R).\] Miller--Nagpal--Patzt proved that, for $F$ a field, $\St(F)$ is Koszul in the sense that $\Tor^{\St(F)}_i(\Z,\Z)_n$ vanishes for $i \neq n$. They asked what the Koszul dual of $\St(F)$ is \cite[Question 3.11]{MNP} (in other words, what is $\Tor^{\St(F)}_n(\Z,\Z)_n$?). Since $\Tor$ groups can be computed using bar constructions, and
$$ \widetilde{H}_{2n-3}(T_n(F) * T_n(F) ) \cong \St_n(F) \otimes \St_n(F),$$ 
 \autoref{mainLemma} lets us answer this question.


\begin{theorem} \label{KD} Let $n \geq 1$. Let $R$ be a PID such that $T_n^2(R)$ is $(2n-4)$-connected.  Then there is an isomorphism of  $\GL_n(R)$-representations 
$$\Tor^{\St(R)}_i(\Z,\Z)_n \cong\left\{\begin{array}{ll}   \widetilde{H}_{2n-3}\left(T_n^2(R)\right), & i=n  \\ 0, & i \neq n.  \end{array} \right. .$$ 
In particular, these isomorphisms hold for $R=F$ a field, and imply isomorphisms of $\GL_n(F)$-representations $$\Tor^{\St(F)}_n(\Z,\Z)_n \cong \St_n(F) \otimes \St_n(F).$$ 
\end{theorem}

\begin{question}
The Koszul dual of an associative ring naturally has a co-algebra structure. Since $\St(R)$ is graded-commutative, there is also an algebra structure on the Koszul dual. How can these structures be described explicitly in terms of apartments? 
\end{question}

\begin{remark}
A $k$-fold iterated bar construction computes $E_k$-indecomposables in the sense of Galatius--Kupers--Randal-Williams \cite{e2cellsI} (also known as $E_k$-homology or $E_k$-Andr\'e--Quillen homology). From this viewpoint, \autoref{mainLemma} implies that $H_*(D^{j+k}(R))$ is the $E_k$-homology of $D^j(R)$. Since the reduced chain complex  $ \widetilde{C}_*(D^1_n(R))$ is equivalent to a shift of $\St_n(R)$, we can interpret $H_*(D^{k}(R))$ as the $E_{k-1}$-homology  of $\St(R)$. In particular, the homology of Rognes' common basis complex measures the $E_\infty$-indecomposables of $\St(R)$. When $R=F$ is a field, this implies that the stable Steinberg modules $\St^{\infty}_n(F)$ measure the $E_\infty$-indecomposables of $\St(F).$
\end{remark}

\subsection{Overview} The paper is structured as follows. 

In \autoref{SectionAlgebraicFoundations} we characterize the common basis property of $R$-submodules in terms of an inclusion-exclusion formula. In \autoref{Section-TheBuildings} we study simplicial complexes that interpolate between the higher buildings of Rognes \cite{Rog1} and the split variations of Galatius--Kupers--Randal-Williams \cite{e2cellsIV}. This section uses tools from combinatorial topology to prove certain complexes are homology equivalent. In \autoref{Section-AlgebraicPropertiesHigherTitsBuildings}, we prove the equivalence $D_n^k(R) \simeq \Sigma^{k+1} T_n^k(R)$, and describe a relationship between the split versions of the complexes and a bar construction in the category of representations of the groupoid $\GL(R)$. \autoref{SectionRognesConjecture} uses these results to prove Rognes' connectivity conjecture for fields,  \autoref{ConnectivityThm}. In \autoref{SectionKoszulDual} we prove \autoref{KD}, in particular computing the Koszul dual of the Steinberg monoid for a field. 

\subsection{Acknowledgments} We thank Alexander Kupers who contributed significantly to the this project but declined authorship. We also thank Andrea Bianchi, Benjamin Br\"uck, S\o ren Galatius, Manuel Rivera, John Rognes, and Oscar Randal-Williams for helpful conversations. We thank two anonymous referees for close readings and detailed suggestions to improve the paper. 

This work was done in part at a SQuaRE at the American Institute for Mathematics. The authors are grateful to AIM for their support.

\section{Algebraic foundations}\label{SectionAlgebraicFoundations}

Throughout the paper, we let $R$ denote a PID and (in later sections) $F$ a field. We begin by recalling some basic algebraic properties of modules over a PID. Then we give a characterization of when a collection of submodules has the common basis property in terms of an inclusion/exclusion condition (\autoref{CBPCriterionPID}). This characterization will be used in \autoref{Section-TheBuildings} to prove certain subcomplexes of higher Tits buildings are full. 

\subsection{Review of some algebraic preliminaries}

We summarize some statements about PIDs $R$. The following result is standard (see, for example, Kaplansky \cite{Kaplansky}).

\begin{lemma}  \label{LemmaSplit} Let $R$ be a PID, and let $U$ be an $R$-submodule of $R^n$. The following are equivalent. 
\begin{itemize}
\item There exist an $R$-submodule $C$ such that $R^n = U \oplus C$. 
\item There exists a basis for $U$ that extends to a basis for $R^n$. 
\item Any basis for $U$ extends to a basis for $R^n$. 
\item The quotient $R^n/U$ is torsion-free. 
\end{itemize} 
\end{lemma}

\begin{definition} Let $R$ be a PID, and $U \subseteq R^n$  a submodule satisfying the equivalent conditions of \autoref{LemmaSplit}. Then we call $U$ a \emph{split} submodule. We may also say that $U$ \emph{is  a  summand} or \emph{has a complement} in $R^n$.  
\end{definition} 

We begin with some basic lemmas. 

\begin{lemma} \label{LemmaNestedSplit}  Let $R$ be a PID. Let $W$ be a summand of $R^n$, and let $U \subseteq W$. Then $U$ is a summand of $W$ if and only if it is a summand of $R^n$. 
\end{lemma} 
\begin{proof}  If $U \subseteq W$ and $W \subseteq R^n$ are both split, then we can extend any basis for $U$ to a basis for $W$, then further extend to a basis for $R^n$. We deduce that $U$ is split in $R^n$. Conversely, suppose $U$ is a summand of $R^n$. Then $W/U \subseteq R^n/U$ is torsion-free, so $U$ is a summand of $W$. 
\end{proof}

\begin{lemma} \label{LemmaIntersectionSplit} Let $R$ be a PID,  and let $U,W \subseteq R^n$ be submodules. If $W$ is a summand of $R^n$, then $W \cap U$ is  summand of $U$.  If $U$ and $W$ are both summands of $R^n$, then in addition $W \cap U$ is a summand of $R^n$. 
\end{lemma} 
\begin{proof} Since $W$ is split in $R^n$, the quotient $U / (U\cap W) \cong (U+W)/ W \subseteq R^n / W$ is torsion-free. Thus $U \cap W$ is a split submodule of $U$. If $U$ is split, then $U\cap W$ is split in $R^n$ by  \autoref{LemmaNestedSplit}. 
  \end{proof}


  
  

  
  \begin{example} \label{ExampleIncompatibleLines} \autoref{LemmaIntersectionSplit} states that the collection of split submodules of $R^n$ is closed under finite intersection. Observe that it is not closed under sum. For example, 
$$ U = \mathrm{span} \left( \begin{bmatrix} 1\\ 1 \end{bmatrix} \right)  \qquad \text{and} \qquad W = \mathrm{span} \left( \begin{bmatrix} 1\\-1 \end{bmatrix} \right)$$ are both split $\Z$-submodules of $\Z^2$, but their sum is index $2$ in $\Z^2$. 
  \end{example}

\begin{definition} \label{DefnCompatible} Let $R$ be a PID and $n \geq 1$. We say a collection  of  submodules $\{U_0, U_1, \dots, U_p\}$  of $R^n$ has the \emph{common basis property} or is \emph{compatible} if there exists a basis $B$ for $R^n$ such that each $U_i$ is spanned by a subset of $B$.  We say that two or more collections of submodules are \emph{compatible} with each other if their union is compatible. 

\end{definition} 

Observe that, if submodules $\{U_0, U_1, \dots, U_p\}$ of $R^n$ have the common basis property, each module $U_i$ is necessarily a summand of $R^n$. 

\begin{example}
For example, the lines in $\Z^3$ spanned by $(1,1,0)$ and $(1, -1, 0)$ are not compatible. Split submodules $\{U_0, U_1, \dots, U_p\}$ that form a flag $U_0 \subseteq U_1 \subseteq \dots \subseteq U_p$ in $R^n$ are always compatible by \autoref{LemmaNestedSplit}. By definition any subset of a compatible collection is compatible. 
\end{example}

We note that the question of compatibility for split $R$-modules $U$ and $W$ depends on the ambient space. For instance, $U$ and $W$ are always compatible as submodules of $U+W$. 


\begin{example} We note that for a collection of split submodules $\{U_1, U_2, \dots, U_k\}$ of $R^n$, pairwise compatibility does not imply compatibility. Even when $R$ is a field, if we take three distinct coplanar lines, any two will be compatible but all three will not. When $R$ is not a field there are more subtle phenomena. For example, the $\Z$-submodules of $\Z^3$
$$ U_1 = \mathrm{span} \left( \begin{bmatrix} 1\\ 1 \\0  \end{bmatrix} \right),   \qquad U_2 = \mathrm{span} \left( \begin{bmatrix} 1\\0 \\1  \end{bmatrix} \right),  \qquad U_1 = \mathrm{span} \left( \begin{bmatrix} 0\\ 1 \\1  \end{bmatrix} \right)$$ 
are pairwise-compatible, non-coplanar lines, but the sum $U_1 \oplus U_2 \oplus U_3$ is index $2$ in $\Z^3$. 
\end{example} 



The following lemmas reflect the fact that, for collections of submodules with the common basis property, the operations of sum and intersection simplify to the operations of union and intersection of sets of basis elements. Hence, in this context these operations have better properties than for arbitrary submodules, for example, intersection distributes over sums. The slogan is that collections of modules with the common basis property behave like sets (of basis elements).

The next result is straight-forward. 

\begin{lemma}  \label{CommonBasisIntersection} Let $R$ be a PID. Suppose that $U, W$ are split submodules of a free $R$-module $R^n$. Suppose there exists some basis $B$ of $R^n$ such that $U$ is spanned by a subset $B_U \subseteq B$ and $W$ is spanned by a subset $B_W \subseteq B$. Then $U + W$ is spanned by $B_U \cup B_W$, and $U \cap W$ is spanned by $B_U \cap B_W$. \end{lemma}

In particular,  $U + W$  and $U \cap V$ are summands of $R^n$ and compatible with
$\{U, W \}$.


\begin{lemma} \label{CommonBasisIntersectionSum} Let $R$ be a PID. Suppose that $U,V, W \subseteq R^n$ have the common basis property. Then $$(U+V) \cap W = (U \cap W) + (V \cap W).$$ 
If $U + V = U \oplus V$, then 
$$(U\oplus V) \cap W = (U \cap W) \oplus (V \cap W).$$ 
\end{lemma} 
\begin{proof} Let $B$ be a basis for $R^n$ with subsets $B_U, B_V, B_W$ spanning $U,V,W$ respectively. By \autoref{CommonBasisIntersection}, $(U+V) \cap W$ is spanned by the basis elements $(B_U \cup B_V) \cap B_W$ and $(U \cap W) + (V \cap W)$ is spanned by the same set $(B_U \cap B_W) \cup (B_V \cap B_W)$. 
\end{proof} 

\begin{lemma} \label{CommonBasisClosure}Let $R$ be a PID, and let $\sigma=\{U_0, \dots, U_p\}$ be a collection of split submodules of $R^n$. Let $\overline{\sigma}$ be the union of $\sigma$ with modules constructed by all (iterated) sums and intersections of elements of $\sigma$.
\begin{enumerate}[(i)]
\item \label{i}  The set  $\sigma$ has the common basis property if and only if  $\overline{\sigma}$ does. 
\item  \label{ii} A submodule $U$ is compatible with $\sigma$ if and only if it is compatible with $\overline{\sigma}$. 
\end{enumerate}
In particular, if $\sigma=\{U_0, \dots, U_p\}$ has the common basis property, modules constructed by (iterated) sums and intersections of the modules $U_i$ are necessarily summands in $R^n$ and compatible with $\sigma$. 
\end{lemma} 

\begin{proof} Part (\ref{i}) follows from \autoref{CommonBasisIntersection}.  
 Part  (\ref{ii}) follows from the observation that a common basis for $\{U\} \cup \sigma$ is also a common basis for $\overline{\{U\} \cup \sigma} \supseteq \{U\} \cup \overline{\sigma}$  by Part (\ref{i}), and vice versa. 
\end{proof}

\begin{example} Let $U$, $W$ be split submodules of $R^n$. If $W$ is contained in a complement of $U$, then $U$ is compatible with $W$, and moreover $U$ is compatible with any split submodule of $W$. In general, however, if $U$ is compatible with $W$ it need not be compatible with split submodules of $W$. For example, the $\Z$-span of $(1, 1, 0)$ in $\Z^3$ is compatible with $W=\Z^3$, but not with the split submodule of $W$ spanned by $(1, -1, 0)$.  
\end{example}


The following result will be used in \autoref{ContractibleTits} to compare certain subcomplexes of higher Tits buildings.

\begin{proposition}\label{mapfdefined} Let $R$ be a PID. 
Let $U,V, W \subseteq R^n$ be submodules with the property that $W+V = R^n$ and $U \subseteq W$. Then $U+V$ is a proper submodule of $R^n$ if and only if $U+(W \cap V)$ is a proper submodule of $W$. 
\end{proposition} 

\begin{proof} We prove the contrapositive. Suppose that $U+(W \cap V) =W.$ Then
\begin{align*}
U + V = U + \big((W \cap V)  +V\big) = \big(U +(W \cap V) ) + V = W+V = R^n 
\end{align*}
Conversely, suppose that  $U+V=R^n$. Then, given $w \in W$, we can write $w = u+v$ for some $u\in U, v \in V$. But then $v  =w-u$ must be contained in $W$, so $v \in W \cap V$.  We see that $w \in U+(W \cap V)$, so we conclude that $U+(W \cap V) = W$. 
\end{proof}




\subsection{The common basis property and inclusion-exclusion} 

The first goal of this subsection is to prove that a collection of submodules has the common basis property if and only if they satisfy a splitting condition and their ranks satisfy an inclusion-exclusion formula. 

We will use this characterization of the common basis property to deduce a lemma on how the common basis property interacts with flags of submodules.  This lemma will imply that certain subcomplexes of the higher Tits buildings are full subcomplexes. We will need this fullness result in the Morse theory arguments of \autoref{Section-TheBuildings}.  

\subsubsection{Posets and corank}

\begin{notation}  Let $R$ be a PID. Let $\cU =\{U_1, U_2, \dots, U_k\}$ be a collection of submodules of $R^n$. Let $[k]=\{1, 2, 3, \dots, k\}$. For a subet $S \subseteq [k]$, let $U_S$ denote the intersection
$$ U_S = \bigcap_{i \in S} U_i, \quad U_{\varnothing} = R^n$$
\end{notation} 

We note that $U_S \cap U_T = U_{S \cup T}$ for all $S,T \subseteq [k]$. \\

Consider the following two posets. 
\begin{itemize}
\item The poset $\big(\{ S \,| \, S \subseteq [k]\}, \supseteq\big)$ on subsets $S \subseteq [k]$ ordered by reverse inclusion.
 \item The poset $\big(\{ U_S \, | \, S \subseteq [k]\}, \subseteq\big)$ on $R$-submodules $U_S$ ordered by containment. \end{itemize} 

The map 
\begin{align*}
\Phi_{\cU} \colon  \{ S \,| \, S \subseteq [k]\} & \longrightarrow \{ U_S \, | \, S \subseteq [k]\} \\ 
S & \longmapsto U_S
\end{align*} 
may not be injective. It is, however, order-preserving, as clearly $T \subseteq S$ implies $U_S \subseteq U_T$. 



\begin{definition} Let $R$ be a PID. Let $\cU = \{U_1, U_2, \dots, U_k\}$ be a collection of submodules of $R^n$. On each poset we define a \emph{corank function}, 
\begin{align*}
F_{\cU} \colon  \{ S \,| \, S \subseteq [k]\} & \longrightarrow \R \\ 
S  & \longmapsto \rank_R(U_S) - \rank_R\left(\sum_{S \subsetneq T}  U_T \right)
\end{align*} 
\begin{align*}
G_{\cU} \colon  \{ U_S \, | \, S \subseteq [k]\} & \longrightarrow \R \\ 
U_S  & \longmapsto \rank_R(U_S) - \rank_R\left(\sum_{U_T \subsetneq U_S}  U_T \right)
\end{align*} 
\end{definition} 

We use the convention that the sum over an empty index set is the zero submodule.  The value $F_{\cU}(S)$ is a lower bound on number of elements that must be added to a basis for the submodule $\sum_{S \subsetneq T}  U_T $ to obtain a basis for $U_S$, with equality if  $\sum_{S \subsetneq T}  U_T $ is split in $U_S$. Similarly, the value  $G_{\cU}(U_S)$ is a lower bound on the number of elements necessary to extend a basis for the submodule $\sum_{U_T \subsetneq U_S}  U_T $ to a basis for $U_S$, with equality if it is split.  The corank functions are illustrated in two examples in \autoref{Poset-Example}. 

We remark that  there are equalities $$\sum_{S \subsetneq T}  U_T \quad = \quad  \sum_{i \in [k], i \notin S}  U_{S \cup \{i\}}   \quad = \quad \sum_{i \in [k], i \notin S}  U_{S} \cap U_i   $$
so we could have equally defined $F_{\mathcal{U}}$ in terms of this latter sum. 

 \begin{figure}[h!]
    \begin{subfigure}[t]{.4\textwidth} \hspace{-2em}
\includegraphics[scale=3.3]{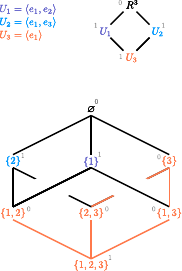}  
      \caption{A collection of submodules of $R^3$ with the common basis property.}
    \end{subfigure} \hfill 
    \begin{subfigure}[t]{.4\textwidth} 
   \includegraphics[scale=3.3]{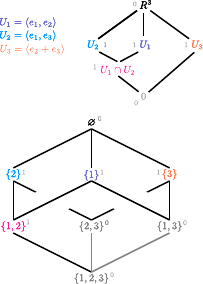} 
      \caption{A collection of submodules of $R^3$ that does not have the common basis property.}
    \end{subfigure}
  \caption{ Let $e_1, e_2, e_3$ denote the standard basis of $R^3$. Two examples are shown of collections of summands $U_1, U_2, U_3$ of $R^3$ and their associated posets. Fibers of $\Phi_{\cU}$ are color-coded according to their image. The values of the corank functions are written next to each element in gray. } 
  \label{Poset-Example}
  \end{figure}
  
  The next lemma shows that the two corank functions capture the same information. However, the function $F_{\cU}$ is technically more useful because its domain is a specific known poset with favourable combinatorial properties (it  is a \emph{boolean lattice}). This will make it easy to apply the Möbius inversion formula, as we will see in the next subsection.

\begin{lemma} \label{FvsG}  Let $R$ be a PID. Let $\cU = \{U_1, U_2, \dots, U_k\}$.
\begin{enumerate}[(i)]
\item \label{preimage} For any $S \subseteq [k]$, the preimage $\Phi_{\mathcal{U}}^{-1}(U_S)$ has a unique minimum element. 
\item \label{LowerSums} For every fixed $S \subseteq [k]$, 
$$ \sum_{S \subsetneq T} U_T = \left\{ \begin{array}{ll} \sum_{U_T \subsetneq U_S} U_T, & \text{$S$ is the unique minimal element in $\Phi_{\mathcal{U}}^{-1}(U_S)$,} \\ U_S, & \text{ otherwise}. \end{array} \right.
$$ 
\item \label{CorankRelationship} For every fixed $S \subseteq [k]$, the corank functions satisfy
$$ F_{\cU}(S) = \left\{ \begin{array}{ll} G_{\cU}(U_S), & \text{$S$ is the unique minimal element in $\Phi_{\mathcal{U}}^{-1}(U_S)$, } \\ 0, & \text{ otherwise}. \end{array} \right.
$$ 
\end{enumerate} 
\end{lemma} 
\begin{proof}  
Consider Part (\ref{preimage}). For fixed $S$,  the preimage $\Phi_{\mathcal{U}}^{-1}(U_S) = \{T \subseteq [k] \; | \; U_T = U_S\}$ contains the unique minimal element $\bigcup_{T \in \Phi_{\mathcal{U}}^{-1}(U_S)} T,$ since the equality $U_T = U_S$ implies
$$ U_S = \bigcap_{T \in \Phi_{\mathcal{U}}^{-1}(U_S)} U_T = U_{\bigcup_{T \in \Phi_{\mathcal{U}}^{-1}(U_S)} T} .$$

For Part (\ref{LowerSums}), fix $S \subseteq [k]$, and consider the sum $\sum_{U_T \subsetneq U_S} U_T$. Given a submodule $U_{T_0} \subsetneq U_S$, we must have $T_0 \neq S$, and $U_{T_0} = U_{T_0} \cap U_S = U_{T_0 \cup S}$. This implies that $U_{T_0}$ is contained in the sum $\sum_{S \subsetneq T} U_T$, and we have containment 
$$\left(  \sum_{U_T \subsetneq U_S} U_T \right) \subseteq \left(\sum_{S \subsetneq T} U_T \right).$$
Conversely, consider a set $T \supsetneq S$. Then $U_{T} \subseteq U_S$. If we have equality $U_{T_0} = U_S$ for some $T_0 \supsetneq S$, then $\sum_{S \subsetneq T} U_T = U_S$. On the other hand, by definition, there is strict containment $U_{T} \subsetneq U_S$ for every $T \supsetneq S$ precisely when $S$ is the minimal element of $\Phi_{\mathcal{U}}^{-1}(U_S)$. In this case, $U_T \subseteq \sum_{U_T \subsetneq U_S} U_T$ for all $T\supsetneq S$, and so we have equality 
$$  \left(\sum_{S \subsetneq T} U_T \right) = \left(  \sum_{U_T \subsetneq U_S} U_T \right).$$
This establishes Part (\ref{LowerSums}). Part (\ref{CorankRelationship}) now follows from Part (\ref{LowerSums}) and the definition of the corank functions. 
\end{proof} 

\begin{corollary} \label{SumF=SumG} $$  \sum_{\{ S \,| \, S \subseteq [k]\}} F_{\cU}(S) =  \sum_{\{ U_S \, | \, S \subseteq [k]\}} G_{\cU}(U_S). $$ 
\end{corollary} 

\begin{proof} By \autoref{FvsG} Part (\ref{CorankRelationship}), the sum $  \sum_{\{ S \,| \, S \subseteq [k]\}} F_{\cU}(S)$  is  the sum $\sum_{\{ U_S \, | \, S \subseteq [k]\}} G_{\cU}(U_S)$ plus additional terms equal to zero. 
\end{proof}

Given a collection of split submodules $\{U_1, \dots, U_k\}$ in $R^n$, suppose we wish to find a minimal collection of spanning vectors of $R^n$ with the property that some subset of these vectors spans $U_i$ for each $i$. The following lemma relates the corank functions to the minimal number of vectors needed to achieve this, and relates this number to the common basis property. 

\begin{lemma} \label{CorankCondition} Let $R$ be a PID. Fix $n$ and let $\mathcal{U}=\{U_1, U_2, \dots, U_k\}$ be a collection of split submodules of $R^n$. 
\begin{enumerate}[(a)]
\item  \label{ItemA} Then the integer $$  \sum_{\{ S \,| \, S \subseteq [k]\}} F_{\cU}(S) =  \sum_{\{ U_S \, | \, S \subseteq [k]\}} G_{\cU}(U_S) $$ is a lower bound on the size of any spanning set $B$ of $R^n$  with the property that each submodule $U_S$, $S \subseteq [k]$, is equal to the span of a subset of $B$. This lower bound will be realized for some spanning set $B$ if for each $S \subseteq [k]$ the module $\sum_{S \subsetneq T}  U_T$ (which equals $\sum_{i \notin S} (U_S \cap U_i)$) is a split submodule of $U_S$. 
\item  \label{ItemC} Let $R$ be any PID.   Then $\mathcal{U}=\{U_1, U_2, \dots, U_k\}$ has the common basis property if and only if
$$  \sum_{\{ S \,| \, S \subseteq [k]\}} F_{\cU}(S) =  \sum_{\{ U_S \, | \, S \subseteq [k]\}} G_{\cU}(U_S) =n $$
and for each $S \subseteq [k]$ the module $\sum_{S \subsetneq T}  U_T$  is a split submodule of $U_S$.  
\end{enumerate} 
\end{lemma} 

See \autoref{Poset-Example} for an illustration of Part (\ref{ItemC}). 

\begin{remark} We collect some remarks about this theorem. 
\begin{itemize} 
\item It is not difficult to check that there is equality $$\sum_{S \subsetneq T}  U_T = \sum_{i \notin S} (U_S \cap U_i).$$ In practice the splitting criterion of \autoref{CorankCondition} part (\ref{ItemC})  can be easier to verify using the second expression.    
\item We proved  the equality $  \sum_{\{ S \,| \, S \subseteq [k]\}} F_{\mathcal{U}}(S) =  \sum_{\{ U_S \, | \, S \subseteq [k]\}} G_{\mathcal{U}}(U_S) $
in \autoref{SumF=SumG}.  
\item The submodules $U_i$ are split by assumption. Thus, for each $S \subseteq [k]$, the intersection of submodules $U_S$ is split by \autoref{LemmaIntersectionSplit}. In fact, by \autoref{CommonBasisClosure}, the submodules $\mathcal{U}$ have the common basis property if and only if the submodules $\{U_S \; | \; S \subseteq [k]\}$ do. 
\item It follows from \autoref{LemmaNestedSplit} that $\sum_{S \subsetneq T}  U_T$ is split in $U_S$ if and only if it is split in $R^n$.  We can simply say $\sum_{S \subsetneq T}  U_T$ ``is split'' without ambiguity. 
\end{itemize} 
\end{remark} 

\begin{remark} \label{RemarkSplit} We note that the splitting criterion of \autoref{CorankCondition} part (\ref{ItemC})  is automatic for fields. This splitting criterion, however, is \emph{not} the obstruction to extending the main results of this paper from fields to general PIDs. Our applications of \autoref{CorankCondition}---and the results in the remainder of this section---will hold for general PIDs.   \end{remark}

\begin{proof}[Proof of \autoref{CorankCondition}] 

We will approach this proof using induction (on height) over the poset of submodules $\{U_S\}$. 

Since it is convenient in this proof to analyze the poset $\{U_S\}$ instead of the poset of subsets of $[k]$, we first reframe the splitting condition in terms of the former poset structure.  \autoref{FvsG} part (\ref{LowerSums}) states that the submodule $\sum_{S \subsetneq T}  U_T$ is equal to $\sum_{U_T \subsetneq U_S} U_T$ or $U_S$.   Thus, the module $\sum_{S \subsetneq T}  U_T$ is split in $U_S$ for every $S\subseteq [k]$ if and only if the submodule $\sum_{U_T \subsetneq U_S} U_T$ is split in $U_S$ for every $S\subseteq [k]$.

Let us call a set $B \subseteq R^n$ a \emph{common spanning set} for $\{U_S\}$ if $B$ spans $R^n$ (as an $R$-module), and a subset of $B$ spans $U_S$ for each $S \subseteq [k]$.  We will construct a minimal (in cardinality) common spanning set $B$  inductively, proceeding by height in the poset $\{U_S\}$ of submodules under containment.  Recall that the height of an element $x$ in a poset is the maximal length $h$ of a chain $x_0 < x_1 < \dots <x_h=x$.  We wish, at step $h$, to build a set $B_{h} \subseteq R^n$ with the property that a subset of $B_{h}$ spans each submodule $U_S$ that has height at most $h$---and to argue that our set $B_{h}$ is minimal in size among all such subsets.  This procedure will, once we reach the height of the submodule $U_{\varnothing}=R^n$, yield a minimal common spanning set of cardinality at least $ \sum_{U_S} G_{\cU}(U_S)$, with equality when the splitting criterion is satisfied. This inductive argument will therefore establish part (\ref{ItemA}) of the lemma. 

For the base case, we choose a basis set $B_0$ for the element $U_{[k]} = \bigcap_i U_i$ of height 0; it may be empty if $U_{[k]}=0$. 
 
We now describe the inductive step, in which we construct the spanning set $B_{h}$ for all modules of height $h$ or less.  Fix a submodule $U_S$ of height $h$. By induction, $B_{h-1}$ already contains a spanning set for $\sum_{U_T \subsetneq U_S} U_T$.  Choose the minimal number of elements of $U_S$ necessary to extend this to a spanning set for $U_S$. If $\sum_{U_T \subsetneq U_S} U_T$ is a split submodule of $U_S$, this can be accomplished with exactly $G_{\cU}(U_S)$ elements of $U_S$. If $\sum_{U_T \subsetneq U_S} U_T$  is not split, it will require more than $G_{\cU}(U_S)$ elements.  We take $B_h$ to be the union of $B_{h-1}$ and these additional spanning elements for each submodule of height $h$. 

We must check that, at each step $h$, the resultant set $B_h$ is minimal in cardinality.  To see this, suppose we had a set $A$ of elements with the property that some subset of $A$ spans every submodule of height $\leq h$. Then $A$ must contain a set of elements spanning all submodules of height $<h$. Moreover, $A$ must contain additional elements that extend the spanning set of  $\sum_{U_T \subsetneq U_S} U_T$ to a spanning set of $U_S$ for each submodule $U_S$ of height $h$. These new sets of elements must be pairwise disjoint:  if some element $v$ is contained in two distinct submodules of height $h$, then $v$ is contained in their intersection, and therefore in a submodule of strictly smaller height in the poset. This proves minimality for $B_h$. 

By construction, the set $B_h$ has cardinality at least $\sum_{\mathrm{height}(U_S) \leq h} G_{\cU}(U_S)$, with equality if  $\sum_{U_T \subsetneq U_S} U_T$ is split for all $U_S$ of height at most $h$. This concludes part (\ref{ItemA}). 

We turn to part (\ref{ItemC}). By \autoref{CommonBasisClosure}, the collection $\mathcal{U}$ has the common basis property if and only if its closure under iterated sums and intersections does. In particular, it is a necessary condition that the sum $\sum_{U_T \subsetneq U_S} U_T$ be split for every $S$. So we assume that the submodule $\sum_{U_T \subsetneq U_S} U_T$ is split for all $S$, and we will show, under this assumption, that $\mathcal{U}$ has the common basis property if and only if $\sum_{U_S} G_{\cU}(U_S) =n$. 

By part (\ref{ItemA}), we have a minimal common spanning set $B$ for $\mathcal{U}$ of cardinality $|B|=\sum_{U_S} G_{\cU}(U_S)$. In particular $B$ is a spanning set for $R^n$, and so it is a basis for $R^n$ if and only if $|B|=n$. If $B$ is not a basis for $R^n$, then $|B|>n$, and by minimality of $B$ no common basis for $\mathcal{U}$ can exist.  The theorem follows. 
\end{proof}


\subsubsection{The inclusion-exclusion property} 
We will use \autoref{CorankCondition} to give a new characterization of the common basis property in terms of the inclusion-exclusion formula. The key tool to deduce this equivalence  is M\"obius inversion.

The M\"obius inversion theorem is due to Rota \cite{Rota}. For this version of the theorem, see (for example) Kung--Rota--Yan \cite[3.1.2. M\"obius Inversion Formula]{KungRotaYan}; we specialize to the case of a finite poset. 

\begin{theorem}[Rota \cite{Rota}] \label{MobiusInversion}
 Let $P$ be a finite poset. Then there exists an invariant of the poset $P$, a function $\mu\colon  P\times P \to \R$, called a  \emph{M\"obius function}, with the following property. If $f$ and $g$ are real-valued functions defined on $P$, then 
$$ f(y) = \sum_{x \leq y} g(x) \; \; \text{ for all $y \in P$}  \qquad \text{if and only if} \qquad g(y) = \sum_{x \leq y} \mu(x,y) f(x)  \; \;  \text{ for all $y \in P$} .$$ 
\end{theorem} 

\begin{proposition}[Rota {\cite[Proposition 3 (Duality)]{Rota}}] \label{RotaDuality}  If $P$ is a finite poset and $P^{\star}$ is its opposite poset, then $\mu_P(x,y) = \mu_{P^{\star}}(y,x)$. 
\end{proposition} 

The following theorem is the statement of Rota's result \cite[Corollary (Principle of Inclusion-Exclusion)]{Rota} combined with  \cite[Proposition 3 (Duality)]{Rota}(here \autoref{RotaDuality}). 
 
\begin{theorem}[Rota \cite{Rota}] \label{MobiusSubsets}   When $P$ is the poset of subsets of a finite set $[k]$ under inclusion, then both $P$ and its opposite poset (ordered by reverse inclusion) have M\"obius function
$$ \mu(S,T) = (-1)^{|S|-|T|}.$$
\end{theorem}

M\"obius inversion generalizes the inclusion-exclusion principle for a collection of sets. We will use this result to relate the inclusion-exclusion formula to the common basis property.  

\begin{definition} \label{DefnInclusionExclusionProperty} Let $R$ be a PID. Fix $n$ and let $\cU = \{U_1, U_2, \dots, U_k\}$ be a collection of  submodules of $R^n$. Recall that, for $S \subseteq [k]$, we write $U_S$ for the intersection $\bigcap_{i\in S} U_i$ and $U_{\varnothing}=R^n$. Then we say that the submodules $\mathcal{U}$  have the \emph{inclusion-exclusion property} if for all $S \subseteq [k]$, 
$$\text{ rank$\left(\sum_{i \notin S} U_S \cap U_i \right)$} = \sum_{T\supsetneq S} (-1)^{|S|-|T| +1} \; \rank(U_T). 
$$ 
Equivalently, for all $S \subseteq [k]$, 
 $$\Big(\text{corank of $\sum_{i \notin S} (U_S \cap U_i) $ in $U_S$}\Big)  = \sum_{T\supseteq S} (-1)^{|S|-|T|} \; \rank(U_T). 
$$ 
\end{definition}

We note that  $ \Big(\text{corank of $\sum_{i \notin S} (U_S \cap U_i) $ in $U_S$}\Big)$ is precisely the quantity $F_{\cU}(S)$.

Observe that, if we specialize \autoref{DefnInclusionExclusionProperty} to the subset $S=\varnothing$, the inclusion-exclusion property states
$$ \rank(U_1 + U_2 + \dots +  U_k) = \sum_{\varnothing \neq T \subset [k]} (-1)^{|T|+1} \; \rank(U_T),$$ equivalently,   
$$F_{\cU}(\varnothing)  = \Big(\text{corank of $(U_1 + U_2 + \dots +  U_k)$ in $R^n$}\Big)  = \sum_{T \subseteq [k]} (-1)^{|T|} \; \rank(U_T).$$

We now prove one of the main results of this section, a characterization of the common basis property in terms of the inclusion-exclusion property. 

\begin{theorem} \label{CBPCriterionPID} Let $R$ be a PID. Fix $n$ and let $\cU = \{U_1, U_2, \dots, U_k\}$ be a collection of split submodules of $R^n$.  The set $\{U_1, U_2, \dots, U_k\}$ has the common basis property if and only if it satisfies the inclusion-exclusion property in the sense  of \autoref{DefnInclusionExclusionProperty} and for each $S \subseteq [k]$ the submodule $\sum_{S \subsetneq T}  U_T$ (which equals $\sum_{i \notin S} (U_S \cap U_i)$) is a summand of $R^n$. 
\end{theorem} 

As noted in \autoref{RemarkSplit}, if $R$ is a field, the splitting condition is automatic. Over a field $F$, a collection of subspaces of $F^n$ has the common basis property if and only it satisfies the inclusion-exclusion property. (This, however, is not the point in our paper where the proofs of the main theorems break down for non-field PIDs). 

\begin{proof}[Proof of \autoref{CBPCriterionPID}] The submodules $\mathcal{U}=\{U_1, U_2, \dots, U_k\}$ have the common basis property if and only if each of their subsets does, which  by \autoref{CorankCondition} holds if and only if
\begin{itemize}
\item $\sum_{i \notin S} (U_S \cap U_i)$  is split for every $S$. \\

\item   $ \displaystyle \; \rank(U_S) = \sum_{S \subseteq T \subseteq [k]} F_{\cU}(T)$ for all $S \subseteq [k]$.
\end{itemize} 
Note that the second item includes the statement that $\displaystyle n =  \sum_{\varnothing \subseteq T \subseteq [k]} F_{\cU}(T)$.

By the M\"obius inversion theorems \autoref{MobiusInversion} and \autoref{MobiusSubsets}, this second condition holds if and only if 
$$ F_{\cU}(S) = \sum_{S \supseteq T} (-1)^{|S|-|T|} \rank(U_T) \qquad \text{for all $ \varnothing \subseteq S \subseteq [k].$}$$
This is precisely the inclusion-exclusion property. 
\end{proof} 

\begin{example} Consider a collection $\{U_1, \dots, U_k\}$ of $R$-submodules of $R^n$. In general, in the statement of \autoref{CBPCriterionPID} it is not enough to replace the inclusion-exclusion property of \autoref{DefnInclusionExclusionProperty} with the single condition that $$ \rank(U_1 + U_2 + \dots +  U_k) = \sum_{\varnothing \neq S \subseteq [k]} (-1)^{|S|+1} \; \rank(U_S). $$ 
Consider, for example, the free $R$-module $R^2 = \langle e_1, e_2\rangle$ and the submodules 
$$ U_1 = R^2, \quad U_2 =  \langle e_1 \rangle,  \quad U_3 =  \langle e_2\rangle,  \quad U_4 =  \langle e_1 + e_2\rangle. $$ 
These submodules fail to have the common basis property. However, the reader can verify that 
$$ \rank(U_1 + U_2 + U_3 +  U_4) = 2 =  \sum_{\varnothing \neq S \subseteq [k]} (-1)^{|S|+1} \; \rank(U_S). $$ 
Here the right-hand-side has nonzero terms 
\begin{align*} 
& \rank(U_1) + \rank(U_2) + \rank(U_3) + \rank(U_4)  - \rank(U_{12}) -  \rank(U_{13}) - \rank(U_{14})  \\ 
& = 2+1+1+1-1-1-1 \\ 
&=2 .
\end{align*} 
In keeping with \autoref{CBPCriterionPID}, the formula 
$$\text{ rank$\left(\sum_{i \notin S} U_S \cap U_i \right)$} = \sum_{T\supsetneq S} (-1)^{|S|-|T| +1} \; \rank(U_T). 
$$ 
does fail for $S=\{1\}$, as the submodules $U_{12},  U_{13},  U_{14}$ do not satisfy the inclusion-exclusion formula. 
\end{example}

\subsubsection{Flags and the common basis property} 

In future sections, we will need the following lemma to prove that certain subcomplexes of our higher Tits buildings are full subcomplexes. This lemma is our main application of \autoref{CBPCriterionPID}.  

\begin{lemma} \label{BigFlagLemma}  Let $R$ be a PID.  Fix $n$. Suppose that $\{U_1, \dots, U_{\ell}\}$ is a collection of submodules of $R^n$, and $V_1 \subseteq V_2 \subseteq \dots \subseteq V_r$ is a flag in $R^n$.  If $\{U_1, \dots, U_{\ell}, V_i\}$ has the common basis properrty for every $1 \leq i \leq r$, then 
$\{U_1, \dots, U_{\ell}, V_1, \dots, V_r\}$ has the common basis property. 
\end{lemma}   

\begin{proof}   The claim is vacuous when $r=1$. We may proceed by induction on the length  $r$ of the flag. To perform the inductive step, it suffices (after reindexing) to consider a flag of length $r=2$. 

Let $r=2$. Assume $\{U_1, \dots, U_{\ell}, V_i\}$ has the common basis property for $i=1, 2$.  For notational convenience set $k={\ell}+2$, so  $U_{k-1} = U_{{\ell}+1}=V_1$ and $U_k = U_{{\ell}+2}=V_2$. 
By \autoref{CBPCriterionPID}, we must verify two things: that $\mathcal{U}=\{U_1, \dots, U_{k-2}, U_{k-1}=V_1, U_k=V_2\}$ has the inclusion-exclusion property, and that for all $S \subseteq [k]$ the submodule $\sum_{i \notin S} (U_S \cap U_i)$  is split. 

We comment that verifying these two properties is a long but conceptually straight-forward computation. For completeness we will write it out explicitly. The calculation will involve direct manipulation of sums of intersections $U_T$ using the following shortlist of elementary observations: 
\begin{itemize}
\item Pairs of $R$-submodules $M,N$ of $R^n$ always satisfy the inclusion-exclusion formula; 
$$ \mathrm{rank}(M+N) =  \mathrm{rank}(M) +  \mathrm{rank}(N) -  \mathrm{rank}(M\cap N).$$ 
\item Given any $R$-submodules $M,N,P$, $$ \text{if $M \subseteq N$ then } (M+P)\cap N = M + (P \cap N).$$ 
\item If submodules $\{M_1, \dots, M_{\ell}, N\}$ have the common basis property, by \autoref{CommonBasisIntersectionSum}, $$(M_1 \cap N) + \dots + (M_{\ell} \cap N) = (M_1 + \dots  +M_{\ell}) \cap N.$$  
\item If a collection of $R$-submodules has the common basis property then all sums of intersections of these modules are split; see  \autoref{CommonBasisClosure}.
\item By assumption, the sets $\{U_1, \dots, U_{\ell}, V_i\}$ satisfy the common basis property for $i=1,2$.
\item By assumption, $V_1 \subseteq V_2,$ in particular $V_1 \cap V_2 = V_1$.
\end{itemize} 

To verify both the inclusion-exclusion property and the split condititon, we first observe that each subset $S \subseteq [k]$ falls into one of four cases. Let $A = S \cap [k-2]$. 
\begin{enumerate}
\item[(i)] $S = A \cup \{k-1, k\}$, i.e.,the intersection $U_S = U_A \cap V_1 \cap V_2 = U_A \cap V_1$. \\  Otherwise said, $U_{A \cup \{k-1, k\}} = U_{A \cup \{k-1\}}$. 

\item[(ii)]  $S=A \cup \{k-1\}$, i.e. $U_S = U_A \cap V_1$. \\
As above, $U_S = U_A \cap V_1 \cap V_2  = U_{S \cup \{k\}}$. 

\item[(iii)]  $S = A \cup \{k\}$, i.e., $U_S = U_A \cap V_2$.

\item[(iv)]  $S=A$, i.e., the intersection $U_S$ has neither factor $V_1$ or $V_2$.
\end{enumerate} 

First we check the inclusion-exclusion property.  
For the set $S = A \cup \{k-1, k\}$, morally, the inclusion-exclusion formula holds because we can delete $U_k=V_2$ from every intersection and appeal to the inclusion-exclusion formula associated to the set $\{U_1, \dots, U_{k-2}, V_1\}$ which by assumption satisfies the common basis property.  In detail, 

 \allowdisplaybreaks
\begin{align*} 
&  \text{ rank$\left(\sum_{i \notin S} U_S \cap U_i \right)$} \\ 
&=  \text{ rank$\left(\sum_{i \in [k-2]\setminus A} U_A \cap V_1 \cap V_2 \cap U_i \right)$} \\
& =  \text{ rank$\left(\sum_{i \in [k-2]\setminus A} U_A \cap V_1  \cap U_i \right)$} \\
& =  \text{ rank$\left(\sum_{i \in [k-1]\setminus (A\sqcup \{k-1\})} U_{A \sqcup \{k-1\} }  \cap U_i \right)$} \\
& = \sum_{ [k-1] \supseteq T \sqcup \{k-1\} \supsetneq A \sqcup \{k-1\} } (-1)^{(|A|+1)-(|T|+1)+1} \; \rank(U_T \cap V_1) \\
& \qquad \qquad \qquad  \text{(by \autoref{CBPCriterionPID}, since $\{U_1, \dots U_{k-2}, V_1\}$ satisfies inclusion-exclusion),} \\ 
& = \sum_{[k-2] \supseteq T\supsetneq A } (-1)^{|A|-|T|+1} \; \rank(U_T \cap V_1) \\
& = \sum_{[k-2] \supseteq T\supsetneq A  } (-1)^{|A|-|T|+1 } \; \rank(U_A \cap V_1 \cap V_2) \\
& = \sum_{[k] \supseteq T \sqcup \{k-1, k\} \supsetneq A \cup \{k-1, k\} } (-1)^{|A|-|T|+1 } \; \rank(U_A \cap V_1 \cap V_2) \\
& = \sum_{[k] \supseteq T' \supsetneq S } (-1)^{|S|-|T'|+1} \; \rank(U_{T'}).
\end{align*}

 For the set $S =A \cup \{k-1\}$, morally, the inclusion-exclusion expression for the corank $F_{\mathcal{U}}(S)$ of $\sum_{i \notin S} (U_S \cap U_i) $ in $U_S$ vanishes since we have cancellation between the terms for $T$ and $T \cup \{k\}$ for each index $T \subseteq [k-1]$. Correspondingly, the corank $F_{\mathcal{U}}(S)$ does indeed vanish since  $U_S \cap U_k = U_S$. In detail, 
 

 \begin{align*}
& \sum_{[k] \supseteq T\supseteq S} (-1)^{|S|-|T|} \; \rank(U_T) \\ 
 = &  \left( \sum_{[k-1] \supseteq T\supseteq (A  \cup \{k-1\})} (-1)^{|A|+1-|T|} \; \rank(U_{T}) \right) + \left( \sum_{[k-1] \supseteq T\supseteq (A  \cup \{k-1\})} (-1)^{|A|+1-|T|-1} \; \rank(U_{T\cup \{k\}}) \right)  \\ 
  = &  \left( \sum_{[k-1] \supseteq T\supseteq (A  \cup \{k-1\})} (-1)^{|A|+1-|T|} \; \rank(U_{T}) \right) + \left( \sum_{[k-1] \supseteq T\supseteq (A  \cup \{k-1\})} (-1)^{|A|+1-|T|-1} \; \rank(U_{T}) \right)  \\ & \qquad \qquad \qquad  \text{ since $U_T = U_{T\cup \{k\} }$, } \\ 
    = &   \sum_{[k-1] \supseteq T\supseteq (A  \cup \{k-1\})} (-1)^{|A|+1-|T|} \; ( \rank(U_{T}) - \rank(U_T) )\\
    =&\, 0 \\
    =& \Big(\text{corank of $\sum_{i \notin S} (U_S \cap U_i) $ in $U_S$}\Big)\\ 
     = & F_{\mathcal{U}}(S). 
 \end{align*}

  For the set $S = A \cup \{k\}$,  
  
   \begin{align*} &  \text{ rank$\left(\sum_{i \notin S} U_S \cap U_i \right)$} \\ 
&=  \text{ rank$\left((U_A   \cap V_2 \cap V_1) +  \sum_{i \in [k-2]\setminus A} U_A  \cap V_2  \cap U_i \right)$} \\
& =  \text{ rank}\left(U_A  \cap V_1 \cap V_2 \right) +  \text{ rank}\left(\sum_{i \in [k-2]\setminus A} U_A  \cap V_2 \cap U_i \right) - \text{ rank} \left( U_A  \cap V_1 \cap V_2 \cap \left(\sum_{i \in [k-2]\setminus A} U_A  \cap V_2 \cap U_i \right)\right)   \\
& \qquad \qquad \qquad  \text{since rank$(C+ B)=$ rank$(C) + $ rank$(B) - $ rank$(C \cap B)$, } \\ 
& =  \text{ rank}\left(U_A  \cap V_1 \cap V_2 \right) +  \text{ rank}\left(\sum_{i \in [k-2]\setminus A} U_A  \cap V_2 \cap U_i \right) - \text{ rank} \left( U_A  \cap V_1\cap V_2  \cap \left(\sum_{i \in [k-2]\setminus A} U_A    \cap U_i \right)\right)   \\
& \qquad \qquad \qquad  \text{because $\sum_{i \in [k-2]\setminus A} U_A  \cap V_2 \cap U_i  =  V_2 \cap \left( \sum_{i \in [k-2]\setminus A} U_A \cap U_i \right)$ by \autoref{CommonBasisIntersectionSum} } \\  & \qquad \qquad \qquad  \text{since $\{U_1, \dots, U_{k-2}, V_2\}$ has the common basis property, } \\ 
& =  \text{ rank}\left(U_A  \cap V_1 \cap V_2 \right) +  \text{ rank}\left(\sum_{i \in [k-2]\setminus A} U_A  \cap V_2 \cap U_i \right) - \text{ rank} \left( U_A  \cap V_1   \cap \left(\sum_{i \in [k-2]\setminus A} U_A    \cap U_i \right)\right)   \\
& =  \text{ rank}\left(U_A  \cap V_1 \cap V_2 \right) +  \text{ rank}\left(\sum_{i \in [k-2]\setminus A} U_A  \cap V_2 \cap U_i \right) - \text{ rank} \left(\sum_{i \in [k-2]\setminus A} U_A    \cap V_1 \cap U_i \right)   \\
& \qquad \qquad \qquad  \text{ by \autoref{CommonBasisIntersectionSum} since $\{U_1, \dots, U_{k-2}, V_1\}$ has the common basis property, } \\ 
& =  \text{ rank}\left(U_A  \cap V_1 \cap V_2 \right)  +  \sum_{[k-2]\cup\{k\} \supseteq T\supsetneq A \cup \{k\}} (-1)^{|A|-|T|} \; \rank(U_T) -  \sum_{[k-1] \supseteq T'\supsetneq A \cup\{k-1\}} (-1)^{|A|-|T'| } \; \rank(U_{T'})\\
& \qquad \qquad \qquad  \text{since \autoref{CBPCriterionPID} applies to the second and third term,} \\ 
& =  \text{ rank}\left(U_A  \cap V_1 \cap V_2 \right)  +  \sum_{[k-2]\cup\{k\} \supseteq T\supsetneq A \cup \{k\}} (-1)^{|A|-|T|} \; \rank(U_T) -  \sum_{[k] \supseteq T\supsetneq A \cup\{k-1, k\}} (-1)^{|A|-|T|+1 } \; \rank(U_T)\\
& \qquad \qquad \qquad  \text{since $U_{T'} = U_{T'\cup \{k\}}$ when $k-1 \in T'$, } \\ 
& =  \sum_{[k] \supseteq T\supsetneq A \cup \{k\}} (-1)^{|A|-|T|} \; \rank(U_T) \\
& =  \sum_{[k] \supseteq T\supsetneq S} (-1)^{|S|-|T|+1} \; \rank(U_T).  
  \end{align*}


 For the set $S=A \subseteq [k-2]$, morally, all intersections involving $V_1$ cancel, and the result follows since  $\{U_1, \dots, U_{k-2}, V_2\}$ satisfies the common basis property. In detail, 
    \begin{align*}
 & \sum_{[k] \supseteq T\supseteq S} (-1)^{|S|-|T|} \; \rank(U_T)  \\ 
  & =   \sum_{[k-2] \supseteq T\supseteq S} (-1)^{|S|-|T|} \; \rank(U_{T})  +  \sum_{[k-2] \supseteq T\supseteq S} (-1)^{|S|-|T|+1} \; \rank(U_{T \cup \{k\}})  \\ & \qquad +  \sum_{[k-2] \supseteq T\supseteq S} (-1)^{|S|-|T|+1} \; \rank(U_{T \cup \{k-1\}})  +  \sum_{[k-2] \supseteq T\supseteq S} (-1)^{|S|-|T|+2} \; \rank(U_{T \cup \{k-1, k\}}) \\ 
    & =   \sum_{[k-2] \supseteq T\supseteq S} (-1)^{|S|-|T|} \; \rank(U_{T})  +  \sum_{[k-2] \supseteq T\supseteq S} (-1)^{|S|-|T|+1} \; \rank(U_{T \cup \{k\}})  \\
    & \qquad \qquad  \text{since $U_{T \cup \{k-1, k\}} = U_{T \cup \{k-1\}}$ for all $T \subseteq [k-2]$}, \\ 
    & = \Big(\text{corank of $U_S \cap V_2 + \sum_{i \in [k-2]\setminus S} (U_S \cap U_i) $ in $U_S$}\Big) \\ & \qquad \qquad \text{since $\{U_1, \dots U_{k-2}, V_2\}$ has the common basis property hence satisfies inclusion-exclusion}, \\ 
        & = \Big(\text{corank of $U_S \cap V_2 + U_S \cap V_1 + \sum_{i \in [k-2]\setminus S} (U_S \cap U_i) $ in $U_S$}\Big)  \\ & \qquad \qquad \text{since $U_S \cap V_1 \subseteq U_S \cap V_2$}, \\
        &= \Big(\text{corank of $\sum_{i \in [k]\setminus S} (U_S \cap U_i) $ in $U_S$}\Big) 
  \end{align*}

%

 Next we check that for all $S \subseteq [k]$ the submodule $\sum_{i \notin S} (U_S \cap U_i)$ is split.  Let $A \subseteq [k-2]$ and consider the four cases. 
Take the case $S=A \cup \{k-1, k\}$. Then $$U_S = U_{A \cup \{k-1,k\} } = U_A \cap V_1 \cap V_2 = U_A \cap V_1,$$  so 
 \begin{align*}
 \sum_{i \notin S} U_S \cap U_i  = \sum_{ i \in [k-2] \setminus A } U_A \cap V_1 \cap U_i.
 \end{align*}
 This sum is split by \autoref{CommonBasisClosure} since $\{U_1, \dots, U_k, V_1\}$ has the common basis property. 
 
 In the case $S=A \cup \{k-1\}$, we see $U_S = U_{A \cup \{k-1\}} = U_A \cap V_1$, so 
 \begin{align*}
 \sum_{i \notin S} U_S \cap U_i  & = \left( \sum_{ i \in [k-2]\setminus A} U_A \cap V_1 \cap U_i \right) + (U_A \cap V_1 \cap V_2) \\  
 & = \left( \sum_{ i \in [k-2]\setminus A} U_A \cap V_1 \cap U_i \right) + (U_A \cap V_1) \\
 & = U_A \cap V_1
 \end{align*}
 which is split since the modules $U_1, \dots, U_k, V_1$ are split (\autoref{LemmaIntersectionSplit}).

  In the case $S=A \cup \{k\}$ we have $U_S = U_{A \cup \{k\}} = U_A \cap V_2$, so 
 \begin{align*}
 \sum_{i \notin S} (U_S \cap U_i)  & = \left( \sum_{ i \in [k-2]\setminus A} U_A \cap V_2 \cap U_i \right) + (U_A \cap V_1 \cap V_2)  \\  
& = \left(\left( \sum_{ i \in [k-2]\setminus A} U_A  \cap U_i \right)\cap V_2\right) + (U_A \cap V_1  \cap V_2) \\ & \qquad \qquad \qquad \qquad \text{ by \autoref{CommonBasisIntersectionSum} since $\{U_1, \dots, U_{k-2}, V_2\}$ has the common basis property}, \\  
& = \left(\left( \sum_{ i \in [k-2]\setminus A} U_A  \cap U_i \right) + (U_A \cap V_1) \right) \cap V_2  \\  & \qquad \qquad \qquad \qquad  \text{ since $U_A \cap V_1 \subseteq V_1 \subseteq V_2$.} 
 \end{align*}
But $\left(\left( \sum_{ i \in [k-2]\setminus A} U_A  \cap U_i \right) + (U_A \cap V_1) \right)$ is split  by \autoref{CommonBasisClosure}  since  $\{U_1, \dots, U_k, V_1\}$ has the common basis property.  Since $V_2$ is split, the intersection is split by \autoref{LemmaIntersectionSplit}.

  In the case $S=A$, 
 \begin{align*}
 \sum_{i \notin S} (U_S \cap U_i)  & = \left( \sum_{ i \in [k-2]\setminus A} U_A \cap U_i \right) + (U_A \cap V_1) + (U_A \cap V_2) \\  
 & = \left( \sum_{ i \in [k-2]\setminus A} U_A  \cap U_i \right)  + (U_A \cap V_2)
 \end{align*}
 since $(U_A \cap V_1) \subseteq (U_A \cap V_2)$. The resultant sum is split by \autoref{CommonBasisClosure} because $\{U_1, \dots, U_k, V_2\}$ has the common basis property. 
\end{proof}

\section{Comparison of higher Tits buildings} \label{Section-TheBuildings}

In this section, we recall the definitions of the classical Tits building and Charney's split Tits building \cite{Charney}.  We then introduce ``higher" Tits buildings generalizing constructions of Rognes \cite{Rog1} and Galatius--Kupers--Randal-Williams \cite{e2cellsIV}. The main result of this section, \autoref{SplitvsNotSplit}, is a comparison theorem between higher buildings with different splitting data.

\subsection{Simplicial complex models of higher buildings}

In this subsection we review the definitions of the (split) Tits buildings and introduce the higher variants. 

\begin{definition}[Tits building] 
Let $M$ be a finite-rank free $R$-module with $R$ a PID. Let $T(M)$ be the realization of the poset of proper nonzero split submodules of $M$ ordered by inclusion.  
\end{definition} 

In this context, we say that proper nonzero split submodules $U,V \subseteq M$ are \emph{comparable} if  $U \subseteq V$ or $V \subseteq U$. 


Charney introduced the following simplicial complex to prove homological stability for general linear groups of Dedekind domains \cite{Charney}.

\begin{definition}[Split Tits building]
Let $M$ be a finite-rank free $R$-module with $R$ a PID. A \emph{splitting of $M$ of size $k$}  is tuple $(V_1,\dots,V_k)$ with $M=V_1 \oplus \dots \oplus V_k$. Let $ST(M)$ be the realization of the poset of splittings of size $2$ with $(P,Q)<(P',Q')$ if $P \subsetneq P'$ and $Q' \subsetneq Q$. We call $ST(M)$ the \emph{split Tits building} associated to $M$. 
\end{definition}

We note that the name ``split Tits building" is intended to emphasize its relationship to the classical Tits building; however, these complexes do not satisfy the combinatorial definition of a building \cite[Definition 4.1]{AbramenkoBrown}. The split Tits buildings appear in Galatius--Kupers--Randal-Williams \cite{e2cellsI, e2cellsIII, e2cellsIV} and Hepworth \cite{Hepworth-Edge}  as instances of  \emph{$E_1$-splitting complexes}  and \emph{splitting posets}.

\begin{proposition} \label{SplittingtVsSplitFlag} 
There is a bijection between the set of $p$-simplices of $ST(M)$ and the set of splittings of $M$ of size $p+2$ given by:
{\footnotesize
\begin{align*} (P_0,Q_0) < \dots <(P_p,Q_p) &\longmapsto (P_0,P_1 \cap Q_0,P_2 \cap Q_1, \dots, P_p \cap Q_{p-1},Q_p ).\\
(M_0, M_1 \oplus \dots \oplus M_{p+1}) 
< \dots <(M_0  \oplus \dots \oplus M_p,M_{p+1}) & \longmapsfrom (M_0, M_1, \dots, M_{p+1} ).
\end{align*}
}
\end{proposition}


See (for example) Hepworth \cite[Proposition 3.4]{Hepworth-Edge}. We now recall the definition of Rognes' common basis complex \cite{Rog1}.

\begin{definition}[Common basis complex] Let $R$ be a PID. 
 The \emph{common basis complex} ${\CB}_n(R)$ is the simplicial complex whose vertices $V$ are proper nonzero summands of $R^n$, and where a set of vertices $\sigma=\{V_0,\dots, V_p\}$ spans a simplex if and only if it has the common basis property in the sense of \autoref{DefnCompatible}. 
\end{definition}



\begin{definition}[Higher Tits buildings] \label{DefnTab}
Let $M$ be a finite-rank free $R$-module with $R$ a PID.  Let $\sigma=\{W_0,\dots,W_k \}$ be a collection of  submodules of $M$ with the common basis property. We define $T^{a,b}(M,\sigma)$, the higher Tits building relative $\sigma$, as a subcomplex of the join $$\underbrace{T(M) * \dots * T(M)}_{a\ \text{times}} * \underbrace{ST(M) * \dots * ST(M)}_{b\ \text{times}}.$$ 
Let $ (V_{0,i} < \dots < V_{n_{i},i}  ) $ denote a flag of proper nonzero submodules, representing a simplex in the $i$th factor $T(M)$. Let $(M_{0,i},\dots,M_{m_{i},i})$ denote a splitting of $R^n$ of size $(m_i+1)$ into proper nonzero submodules, representing a simplex in the $i$th factor $ST(M)$. Then a simplex  
$$ (V_{0,1} < \dots < V_{n_1,1}  ) * \dots * (V_{0,a} < \dots < V_{n_a,a}  ) * (M_{0,1},\dots,M_{m_1,1}  ) * \dots * (M_{0,b},\dots,M_{m_b,b}  ) $$ 
of the join is contained in $T^{a,b}(M,\sigma)$ precisely when $$\{V_{0,1}, \dots, V_{n_a,a}, M_{0,1}, \dots,   M_{m_{b,b}} , W_0, \dots, W_k  \}$$ has the common basis property. For $\sigma$ empty, denote $T^{a,b}(M,\sigma)$ by $T^{a,b}(M)$. We write $T(M, \sigma)$ for $T^{1,0}(M, \sigma)$ and $ST(M, \sigma)$ for $T^{0,1}(M, \sigma)$. 
\end{definition}
Notably, in the definition, the submodules  appearing in different join factors can coincide with each other or with the submodules $W_i$. 

\begin{example} Let $M$ be a finite-rank free $R$-module. A \emph{frame} for $M$ is a decomposition $M= L_1\oplus L_2 \oplus \dots \oplus L_n$ of $M$ into a direct sum of lines. If $\sigma$ is a fixed frame $\{L_1, \dots, L_n\}$, then $T(M, \sigma) \cong S^{n-2}$  is the full subcomplex of $T(M)$ on vertices spanned by proper nonempty subsets of $\sigma$. This is called the \emph{apartment} associated to the frame $\sigma$. 
\end{example} 


In the following example, we will see a structural difference in the relative buildings  $T(M, \sigma)$ when $R$ is a field or a non-field PID.  This distinction is essentially the reason we only prove our main theorem, \autoref{ConnectivityThm}, in the case that $R$ is a field. 

\begin{example} \label{FieldVSPID} Let $R=F$ be a field, and let $M$ be a finite-dimensional $F$-vector space. Let $\sigma$ be a flag in $M$. Then $T(M, \sigma) = T(M)$.  Since $T(M)$ is a \emph{building}, by definition any two simplices are contained in a common \emph{apartment}, i.e., any two flags in $M$ admit a common basis for $M$. 
See for example Abramenko--Brown \cite[Section 4]{AbramenkoBrown} for precise definitions and \cite[Section 4.3]{AbramenkoBrown} for a proof. 

This is not true for more general rings $R$. For example, let $M=R^2 = \langle e_1, e_2\rangle.$  If  $\sigma = \{ R e_1 \}$ then $T(R^2, \sigma)$ has vertices all spans of primitive vectors with second coordinate a unit or zero. If $R$ is not a field, this is strictly smaller than the vertex set of $T(R^2)$. 
\end{example}



The following lemma is a consequence of \autoref{BigFlagLemma}. 

\begin{lemma} \label{RelativeTitsFull} Let $M$ be a finite-rank free $R$-module, and $\sigma$ a set of proper nonzero submodules with the common basis property. Then $T(M, \sigma)$ is a full subcomplex of $T(M)$. 
\end{lemma}

\subsection{Combinatorial Morse Theory}
To analyze these complexes we will employ variations on a technique sometimes called combinatorial Morse theory. See Bestvina \cite{Bestvina-MorseTheory} for exposition on this tool. This method has been used (for example) by Charney \cite{Charney} in her analysis of split Tits buildings.

\begin{theorem}{\bf (Combinatorial Morse Theory).}  \label{Morse} Let $X$ be a simplicial complex and let $Y$ be a full subcomplex of $X$. Let $S$ be the set of  vertices of $X$ that are not in $Y$. Suppose there is no edge between any pair of vertices $s,t  \in S$. Then $$X/Y \simeq  \bigvee_{s \in S} \Sigma \big( \Link_X(s) \big).$$ In particular, if $Y$ is contractible and $  \Lk_X(s) $ is contractible for every $s \in S$, then $X$ is contractible. 

\end{theorem}

\autoref{Morse} allows us to construct the complex $X$ out of the complex $Y$ and the set $S$ of added vertices. For each vertex $s \in S$ we cone off the subcomplex $\Link_X(s) = \Link_X(s)  \cap Y$ in $Y$ and identify $s$ with the cone point. 

We now use \autoref{Morse} to prove \autoref{Morse2}, a more general version of combinatorial Morse theory that allows us to add higher-dimensional simplices as well as vertices.  Concretely, we `add' a set $S$ of new simplices $\sigma$ to a complex $Y$ in the sense that $Y$ is embedded into a complex $X = Y \cup \bigcup_{\sigma \in S} \Star_X(\sigma)$  satisfying $\Star_X(\sigma) \cap Y = \partial \sigma * \Lk_X(\sigma)$ for each $\sigma \in S$. 

\autoref{Morse2} is closely related to the `link arguments' (also called `badness arguments') standard in the field of homological stability; see for example Hatcher--Vogtmann \cite[Proposition 2.1 and Corollary 2.2]{HatcherVogtmann-Tethers}. 

\begin{theorem}{\bf (Generalized Combinatorial Morse Theory).}  \label{Morse2} Let $X$ be a simplicial complex and let $Y$ be a  subcomplex of $X$. Let $S$ be a set of  simplices of $X$ satisfying the following conditions:
\begin{enumerate}[(i)]
\item \label{Item-Morse2i} Suppose $\sigma$ is a simplex in $X$. Then $\sigma$ is a simplex in $Y$ if and only if no face of  $\sigma$ is in $S$.
\item \label{Item-Morse2ii} If $\tau$ and $\sigma$ are distinct simplices in $S$, then the union of their vertices does not form a simplex in $X$. 
\end{enumerate}
Then $$X/Y \simeq  \bigvee_{\sigma \in S} \Sigma^{\dim(\sigma)+1} \big( \Link_X(\sigma) \big).$$ 
\end{theorem}
Condition (\ref{Item-Morse2ii}) implies that any simplex of $X$ can have at most one face in $S$. By Condition (\ref{Item-Morse2i}), the simplices with no face in $S$ are precisely the simplices in $Y$. The subcomplex $Y$ is determined by the set $S$.  

\begin{proof}[Proof of \autoref{Morse2}] We will reduce this statement to \autoref{Morse}. We will construct a new simplicial complex $X'$ homeomorphic to $X$ by subdividing $X$. Heuristically, we will introduce a new vertex $v(s)$ at the barycenter of each positive-dimensional simplex $s \in S$, and take the minimal necessary subdivision. Condition (\ref{Item-Morse2ii}) will ensure a canonical minimal subdivision exists.  An example of complexes $X$ and $X'$ is shown in \autoref{MorseExample1}. 

\begin{figure}[h!]
\includegraphics[scale=3]{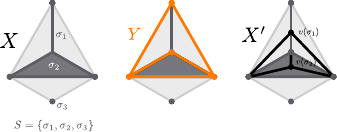} 
\caption{In this example, the set $S$ consists of the edge $\sigma_1$, the 2-simplex $\sigma_2$, and the vertex $\sigma_3$ shown in dark gray in the first figure. The subcomplex $Y \subseteq X$ is highlighted in orange in the second figure. The subdivided complex $X'$ is shown in the third, with the newly added simplices shown in black.} 
\label{MorseExample1}
\end{figure}

Let $\V(X)$ be the set of vertices of $X$. The vertices of $X'$ will be the disjoint union of $\V(X)$ and $S' = \{ v(s)  \mid  s \in S \text{ with } \dim(s)\geq 1\}$.  The following is a complete list of simplices of $X'$: 
\begin{enumerate}[(a)]
\item \label{SimplicesType(a)} A set of vertices $x_0, \dots, x_p \in \V(X)$ spans a simplex in $X'$ if and only if they span a simplex in $X$ that has no positive-dimensional face contained in $S$.
\item \label{SimplicesType(b)} Let $s =[s_0, \dots, s_q]$ be a positive-dimensional simplex in $S$. Then for  $x_0, \dots, x_p \in \V(X)$, the vertices $x_0, \dots, x_p, v(s)$ span a simplex in $X'$ if and only if the vertices in the (not necessarily disjoint) union  $\{s_0, \dots, s_q\} \cup \{x_0, \dots, x_p \}$ form a simplex in $X$ but $\{s_0, \dots, s_q\} \not\subseteq \{x_0, \dots, x_p \}$. 
\end{enumerate} 
We briefly verify that our construction of $X'$ is well-defined, by checking that all faces of simplices in $X'$ are themselves simplices. This is straightforward except for a simplex $[x_0, \dots, x_p, v(s)]$ of type (\ref{SimplicesType(b)}) and a face $\alpha=[x_{i_0}, \dots, x_{i_t}]$ not containing the vertex $v(s)$. By assumption $\alpha$ is a face of the simplex spanned by the vertices $\{s_0, \dots, s_q\} \cup \{x_0, \dots, x_p \}$ but $\alpha$ does not contain $s \in S$. By Condition (\ref{Item-Morse2ii}) $\alpha$ cannot have a face in $S$. Thus $\alpha$ is a simplex of type (\ref{SimplicesType(a)}). 

To argue that $X'$ is homeomorphic to $X$, we interpret $X'$ as a subdivision of $X$. Simplices with no positive-dimensional face in $S$ are not divided; these are simplices of type (\ref{SimplicesType(a)}).  So we consider simplices with faces in $S$. 

 Let $s =[s_0, \dots, s_q]$ be a positive-dimensional simplex of $S$. Let $\sigma$  be a simplex spanned by the {\it disjoint} union of the vertices $\{s_0, \dots, s_q\} \sqcup \{x_{q+1}, \dots, x_r \}$. By Condition (\ref{Item-Morse2ii}), $s$ is the only face of $\sigma$ in $S$.  Item (\ref{SimplicesType(b)})  states that, to build $X'$ from $X$, we replace $\sigma$ with a homeomorphic complex obtained by adding a vertex $v(s)$ to the barycenter of the face $s$ and subdividing $\sigma$ into a union of $(q+1)$ simplices. Specifically, we replace  the simplex $\sigma$ by the union of the simplices spanned by the vertices $x_{q+1}, \dots, x_r, v(s)$ and each proper subset of $\{s_0, \dots s_q\}$.

Observe that (by Condition (\ref{Item-Morse2i})) simplices in $Y \subseteq X$ are not subdivided in $X'$. Thus we may view $Y$ as a subcomplex of $X'$ in a manner compatible with the homeomorphism $X \cong X'$. 
Let $T = S' \cup \{s  \mid   s \in S \text{ with $\dim(s)=0$}\}$ be the union of the new vertices $v(s)$ in $X'$ and the elements of $S$ that are vertices.  Condition (\ref{Item-Morse2i}) states that a vertex of $X$ is in $Y$ if and only if it is not an element of $S$. Hence $T$ is precisely the set of vertices of $X'$ not contained in $Y$. 

The set $T$ is in canonical bijection with $S$. For uniformity we write $v(s)$ for the vertex $s$ when $s$ is an element of $S$, so $T=\{v(s) \mid   s \in S\}$.

We now verify that  there is no edge in $X'$ between any pair of vertices in $T$. From our description of the simplices of $X'$, there are no edges in $X'$ between ``new" vertices $v(s)$ in $S'$. Given a $0$-dimensional simplex $s \in S$, Condition (\ref{Item-Morse2ii})  states that $s$ is not contained in any simplex of $X$ with a face $\sigma$ in $S$ distinct from $s$. This condition implies there are no edges between distinct $0$-simplices of $S$ in $X$, and in our construction of $X'$ no new edges are added between vertices of $X$. Condition (\ref{Item-Morse2ii})  implies moreover, by our description of simplices of type (\ref{SimplicesType(b)}),  that there is no edge between $s$ and $v(\sigma)$ for any positive-dimensional simplex $\sigma \in S$. We deduce that there are no edges between elements of $T$. 

Next we verify that $Y$ is a full subcomplex of $X'$ (specifically, the full subcomplex on vertices not in $T$). Let $x_0, \dots, x_p$ be a set of vertices in $Y$ that span a simplex $\sigma$ in $X'$. Since the vertices $x_i$ are in $Y$ they must 
\begin{itemize}
\item be vertices of $X$, and
\item not include any 0-simplex in $S$. 
\end{itemize} 
The first item implies that $\sigma$ is a simplex of type (\ref{SimplicesType(a)}). Condition (\ref{Item-Morse2i})  and the second item then imply that $\sigma$  is a simplex of $Y$.


We can therefore apply \autoref{Morse} to the complex $X'$, its full subcomplex $Y$, and the vertex set $T$. This implies
$$X/Y \cong X'/Y \simeq  \bigvee_{s \in S} \Sigma\big( \Link_{X'}( v(s))  \big).$$ 
But
$$ \Link_{X'}( v(s))  \cong ( \Link_{X}(s) ) * \partial s \cong  \Sigma^{\dim(\sigma)} \big( \Link_X(\sigma)\big).  $$
The result follows.
\end{proof}

\subsection{Contractible subcomplexes}

In this subsection, we prove that certain subcomplexes of higher Tits buildings are contractible. In the next subsection we will use this result in our proof that the homology of $T^{a,b}(M)$ (\autoref{DefnTab}) depends only on the sum $a+b$ whenever $a \geq 1$. 

\begin{lemma} \label{sigmaClosure}
Let $M$ be a finite-rank free $R$-module with $R$ a PID.  Let $\sigma=\{W_0,\dots,W_k \}$ be a collection of split submodules of $M$. Let $\bar \sigma$ be the closure of $\sigma$ under intersections and sum. Then $T^{a,b}(M,\bar \sigma)= T^{a,b}(M, \sigma)$.
\end{lemma}

\begin{proof} This lemma is a consequence of \autoref{CommonBasisClosure}. 
\end{proof}

The map 
\begin{align*}
ST(M) & \longrightarrow T(M) \\ 
(P,Q) & \longmapsto P
\end{align*}
 that forgets complements induces maps $$T^{a,b}(M,\sigma) \m T^{a',b'}(M,\sigma)$$ whenever $a+b=a'+b'$ and $a' \geq a$. The goal of this section is to prove this map is a homology equivalence whenever $a \geq 1$. The following proposition is the main technical result that we use to prove this comparison.

\begin{proposition} \label{ContractibleTits}
Let $M$ be a finite-rank free $R$-module with $R$ a PID.  Let $\tau=\{V_0, \dots,V_p\}$ be a nonempty simplex of $T(M)$ with $V_0 \subsetneq V_1 \subsetneq \dots \subsetneq V_p$. Let $\sigma$ be a simplex of ${\CB}(M)$ containing $\tau$. Let $\tau_1,\dots,\tau_r$ be distinct simplices of $ST(M,\sigma)$ in the preimage of $\tau$ under the map $ST(M,\sigma) \m T(M,\sigma)$.  Let $\sigma_i$ be the simplex in $\CB(M)$ obtained by taking the union of $\sigma$ and the sets of submodules corresponding to vertices of $\tau_i$. If $r \geq 2$ and $\rank (M) \geq 2$, then $$T(M,\sigma_1) \cap \dots \cap T(M,\sigma_r) $$ is contractible.
\end{proposition}

Let us parse the statement of the proposition. We fix a flag $\tau=\{V_0, \dots,V_p\}$ in $M$ and a set $\sigma = \{V_0, \dots, V_p, Z_0, \dots, Z_k\}$ with the common basis property.  Then $$\tau_i = \{(V_0, C_{0,i}), (V_1, C_{1,i}), \dots, (V_p, C_{p,i})\}$$ and $$\sigma_i = \{C_{0,i} , \dots, C_{p,i}\} \cup \sigma = \{ V_0, C_{0,i} , \dots, V_p,  C_{p,i},Z_0, \dots, Z_k\}  $$ are defined by the conditions:
\begin{itemize}
\item $V_{\ell} \oplus C_{{\ell,i}} = M$ for all $\ell$,
\item $M \supsetneq C_{0,i} \supsetneq C_{1,i} \supsetneq \dots \supsetneq C_{p,i} \supsetneq 0$,
\item The collection of submodules $\sigma_i=\{ V_0, C_{0,i} , \dots, V_p, C_{p,i}, Z_0, \dots, Z_k\} $ has a common basis,
\item The flags $C_{0,i} \supsetneq \dots \supsetneq C_{p,i} $ and $C_{0,j} \supsetneq  \dots \supsetneq C_{p,j} $ are distinct for all $i\neq j$. This means there is at least one index $\ell$ with $C_{\ell, i} \neq C_{\ell, j}$. 
\end{itemize} 
A submodule $W$ is a vertex of $T(M,\sigma_1) \cap \dots \cap T(M,\sigma_r) $ precisely when $W$ is a proper nonzero submodule of $M$ compatible with $\sigma_i =  \{C_{0,i} , \dots, C_{p,i}\} \cup \sigma$ for each $i=1, \dots, r$.  This set of submodules $W$ contains  $\sigma$ and is closed under sum and intersection with the submodules in $\sigma$.

\begin{proof}[Proof of \autoref{ContractibleTits}] This proof is inspired by Charney \cite{Charney} and by the proof of the Solomon--Tits Theorem in Bestvina's notes \cite[Theorem 5.1]{Bestvina-MorseTheory}. 

Let $n=\rank(M)$. We proceed by induction on $n$. Let $$X=T(M,\sigma_1) \cap \dots \cap T(M,\sigma_r). $$ 
By \autoref{BigFlagLemma}, $X$ is a full subcomplex of $T(M)$. In other words, given a flag $U_0 \subseteq U_1 \subseteq \dots \subseteq U_p$ in $M$, if each $U_i$ is compatible with each of $\sigma_1, \dots ,\sigma_r$, the lemma ensures that $\{U_0, \dots, U_p\}$ is compatible with each of $\sigma_1, \dots ,\sigma_r$. Thus a set of vertices in $X$ span a simplex if and only if they form a flag. 

We will proceed by using repeated application of combinatorial Morse theory (\autoref{Morse}).  We will partition the vertices of $X$ into sets $$A_0 \sqcup A_1 \sqcup A_2 \sqcup \cdots \sqcup A_{n-2} \sqcup B_1 \sqcup B_2 \sqcup \cdots \sqcup B_{\rank(V_0)}.$$ This partition defines a filtration of $X$, 
$$  F_0 \; \;  \subseteq \;\;   F_1 \; \;  \subseteq \; \;  \cdots \; \;  \subseteq \; \;  F_{n-2} = G_0 \;\;   \subseteq \; \;  G_1 \; \;   \subseteq \;  \;  \cdots \;\;   \subseteq\; \;  G_{\rank(V_0)} =X$$
where for $k\geq 0$ the complex $F_k$ is the full subcomplex of $X$ on the vertices $A_0  \sqcup \cdots \sqcup A_k$, and for $k \geq 1$ the complex $G_k$ is the full subcomplex on the vertices $A_0 \sqcup \cdots \sqcup A_{n-2} \sqcup B_1 \sqcup B_2 \sqcup \cdots \sqcup B_{k}$. We will apply \autoref{Morse} to each successive step in this filtration. 
The partition of the vertices, and the associated filtration of $X$, are defined as follows.  
\begin{itemize}
\item $A_0$ is the set of vertices $W$ in $X$ comparable to $V_0$ under containment. That is, 
\begin{align*} A_0 & = \left\{ W \;  \middle |\; W \subseteq V_0 \text{ or } V_0 \subseteq W \right\}. \\ 
F_0& =\Star_X(V_0), \text{ the full subcomplex of $X$ on $A_0$.}
\end{align*}
\item The sets $A_1, \dots, A_{n-2}$ are vertices $W$ not comparable to $V_0$ with $V_0 + W \neq M$, further partitioned by rank. Observe that such modules $W$ must have rank at most $(n-\rank(V_0)-1) \leq (n-2)$. Specifically, we define, 
\begin{align*} 
 \quad & \quad  \text{for $k \geq 1$, }   \\ 
 A_k  &=  \left\{ W \;  \middle |\; \begin{array}{l}  W \text{ a vertex of $X$}, \, W \not\in A_0, \\  \rank(W)=k,  \,  W + V_0 \neq M \end{array} \right\}. \\ 
F_k & = \text{  the full subcomplex of $X$} \\& \quad \text{obtained from $F_{k-1}$ by adding the vertices $A_k$.} 
 \end{align*}

\item The sets $B_1, \dots, B_{\rank(V_0)}$ are vertices $W$ not comparable to $V_0$ satisfying $V_0 + W =M$, further partitioned by rank. Specifically, 
\begin{align*} 
G_0 & = F_{n-2} \\ 
  &   \text{ and for $k \geq 1$, }   \\ 
 B_k & =   \left\{ W \;  \middle | \; \begin{array}{l}  W \text{ a vertex of $X$},  \\  \rank(W)=n-k,  \,  W + V_0 = M \end{array} \right\}. \\ 
G_k  & = \text{  the full subcomplex of $X$} \\ & \quad \text{obtained from $G_{k-1}$ by adding the  vertices  $B_k$.}  \end{align*}
\end{itemize}

Observe that $F_0=\Star_X(V_0)$  is contractible by construction. 
We first show, by induction on $k$, that the full subcomplex $F_k$ of $X$ is contractible for each $k=0,1, \dots, n-2$. 
For $k \geq 1$, there are no edges in $X$ between vertices of $A_k$ since they have the same rank as $R$-modules.  By \autoref{Morse}, then, to prove $F_k$ is contractible, it suffices to show that $ \Link_X(W) \cap F_{k-1} $ is contractible for all $W \in A_k$. 

So suppose $A_k$ is nonempty and let $W \in A_k$. Define the shorthand 
$$ \Link^>_X(W)  = \text{the full subcomplex of $X$ on vertices } \{ U \in X  \mid   U \supsetneq W\}  $$ $$ \Link^<_X(W)  = \text{the full subcomplex of $X$ on vertices }\{ U \in X \mid   U \subsetneq W\}. $$ Since $X$ is a full subcomplex of $T(M)$, both subcomplexes are contained in  $\Link_X(W) $, and we may express $\Link_X(W) $ as the join
$$\Link_X(W) =  \Link^>_X(W)  * \Link^<_X(W) .$$ 
Since $F_{k-1}$ is a full subcomplex of $X$, 
$$ \Link_X(W) \cap F_{k-1} = \big( \Link^>_X(W) \cap F_{k-1} \big) * \big(\Link^<_X(W) \cap F_{k-1} \big).  $$
To show that $\Link_X(W) \cap F_{k-1}$  is contractible it then suffices to show $\Link^>_X(W) \cap F_{k-1}$ is contractible. 

We will show that the submodule $(V_0 +W)$ is a cone point of $ \Link^>_X(W) \cap F_{k-1} $.  We first check that $(V_0+W)$ is a vertex in  $F_{k-1}$. Since $V_0$,  $W$, and $\sigma_i$ are compatible for any $i$ by definition of $X$, the sum $(V_0+W)$ is split and compatible with $\sigma_i$. 
Moreover,  $V_0 \subseteq (V_0+W) \neq M$, so we see that $(V_0+W) \in F_0 \subseteq F_{k-1}$ as claimed. Next, since $W \subsetneq (V_0+W)$, we deduce  $(V_0 +W) \in \Link^>_X(W) \cap F_{k-1} $. 

We must verify that, for any vertex $U \in \Link^>_X(W) \cap F_{k-1}$, the modules $U$, $W$, and $(V_0+W)$ form a chain under containment. Because $X$ is a full subcomplex of $T(M)$, this will imply that $(V_0+W)$ is a cone point of $ \Link^>_X(W) \cap F_{k-1} $.

Let $U \in \Link^>_X(W) \cap F_{k-1}$. The statement $U \in \Link^>_X(W)$ means there is containment $W \subsetneq U$, which implies $U \in F_0$ since rank$(U)$ is too large for $U$ to be contained in $A_1, A_2, \dots, A_{k-1}$. Thus, $U$ is comparable to $V_0$ under containment. 
We cannot have containment $U \subseteq V_0$, as then the containment $W \subseteq U \subseteq V_0$ would imply $W \in A_0$. 
Hence $V_0 \subseteq U$, and $W \subseteq (V_0+W) \subseteq U$ as desired.  We conclude that $(V_0+W)$ is a cone point of $\Link^>_X(W) \cap F_{k-1}$, hence $\Link_X(W) \cap F_{k-1}$ is contractible.

Now, $$F_{n-2} = \text{the full subcomplex of $X$ on vertices } \{W \mid  V_0+W \neq M\};$$ this includes the vertices $W\subseteq V_0$ and $V_0 \subseteq W$. We have shown this subcomplex is contractible. We begin our second sequence of applications of \autoref{Morse}. We set $G_0=F_{n-2}$. Recall that we defined, for $k \geq 1$, the set $B_k$ to be those vertices $W$ of $X$ of rank $(n-k)$ satisfying $W+V_0=M$. 
%
%
We denoted by $G_k$ the associated subcomplex in the filtration of $X$.
Let $d=\rank(V_0)$, so $X=G_{d}$. Assume by induction that $G_{k-1}$ is contractible.  Since there are no edges between vertices of $B_k$, we may apply \autoref{Morse}, and show $G_k$ is contractible by verifying $ \Link_X(W) \cap G_{k-1} $ is contractible for all $W \in B_k$. 

Again we may decompose this subcomplex as the join, 
$$ \Link_X(W) \cap G_{k-1} = \big( \Link^>_X(W) \cap G_{k-1} \big) * \big(\Link^<_X(W) \cap G_{k-1} \big),  $$
and we will verify that $\Link_X(W) \cap G_{k-1}$  is contractible by showing  $\Link^<_X(W) \cap G_{k-1}$ is contractible. We first assume that $k<d$. 
Let 
\begin{align*} f\colon \Link^<_X(W) \cap G_{k-1} & \longrightarrow \Link^<_X(W) \cap G_{k-1} \\ 
U & \longmapsto U + (W \cap V_0).  \end{align*}
To check this function is well-defined, we first verify that $U + (W \cap V_0)$ is a vertex of $X$ and contained in $\Link^<_X(W)$. 
We note that $\{U, V_0, W\}$ must be compatible, and compatible with $\sigma_i$ for all $i$, by definition of $X$. Thus $U + (W \cap V_0)$ is split in $M$, and compatible with $\sigma_i$ for all $i$.  We next check that $U + (W \cap V_0) \subsetneq W$, which will establish that $U + (W \cap V_0)$ is a vertex of   $\Link^<_X(W)$. By  \autoref{mapfdefined}, it suffices to show that 
 $U+V_0  \neq M$.  Since $U \subsetneq W$, we know $\rank(U)<\rank(W)$. This implies $U$ cannot be contained in the set $B_i$ for $i \leq k-1$, and $U$ must be contained in one of $A_0, A_1, \dots, A_{n-2}$. It follows from the definition of the sets $A_k$ that  $U+V_0  \neq M$ as claimed. 
 
 We further note that the module $U + (W \cap V_0)$ is a vertex of $G_{k-1}$, as $G_{k-1}$ contains all vertices of $X$ that are submodules of rank at least $\rank(U)$. We have established that $U + (W \cap V_0)$ is a vertex of  $\Link^<_X(W) \cap G_{k-1}$. 

Because $f$ is an order-preserving self-map of a poset, it is a homotopy equivalence onto its image (see Quillen \cite[1.3 ``Homotopy Property"]{Quillen-Poset}). Its image is contractible since $W \cap V_0$ is a cone point. Here we are using the assumption $k<d$ to ensure $W \cap V_0 \neq 0$. We have proved by induction that the subcomplex $G_{d-1} \subseteq X$ is contractible.



%

Finally, consider the case $k=d$. It remains to add the vertices $B_d$ of $X$ that are direct complements to $V_0$ in $M$. Recall this whole proof is by induction on $n=\rank(M)$. We will first verify the case $k=d$ in the base case, $\rank(M)=2$, by showing that in this case, $B_d = \varnothing$. In general if $B_d$ is empty, we have already shown $X=G_{d-1}$ is contractible. 

To carry out the $\rank(M)=2$ case, we will make two general claims that we will wish to invoke again later.  Recall that the complex $X$ was defined by the data
\begin{align*}
\tau  = & \, \{V_0, V_1, \dots, V_p\}, \quad 0 \subsetneq V_0 \subsetneq V_1 \subsetneq \dots \subsetneq V_p \subsetneq M \\
\tau  \subseteq & \,  \sigma,  \\ 
\tau_i  = & \,  \{(V_0, C_{0,i}), (V_1, C_{1,i}), \dots, (V_p, C_{p,i})\}, \quad M \supsetneq  C_{0,i} \supsetneq C_{1,i} \supsetneq \dots \supsetneq C_{p,i}  \supsetneq  0,  \\ 
& V_{\ell} \oplus C_{\ell,i} = M \text{ for all $\ell$} \\ 
\sigma_i =&  \,\sigma \cup \{C_{0,i}, \dots C_{p,i}\}
\end{align*} 
We claim: 
\begin{equation} \label{C0=W} \text{ If $W \in B_d$, then $C_{0,i}=W$ for all $i$.}
\end{equation} 
\begin{equation}\label{Claimp>0} \text{If $\tau$ has dimension $p=0$, then $B_d$ is empty.}
\end{equation} 
To verify Claim (\ref{C0=W}), observe that for any $W \in B_d$, the modules $\{W, V_0, C_{0,i}\}$ must have a common basis for each $i$ by definition of $X$. Since $W$ and $C_{0,i}$ are both direct complements to $V_0$, this is possible only if $W=C_{0,i}$ for each $i$, establishing the claim. We now use this result to verify Claim (\ref{Claimp>0}). Suppose $\tau = \{V_0\}$. By hypothesis, there are at least two distinct preimages $\tau_1 = \{(V_0, C_{0,1}) \}$ and $\tau_2 = \{(V_0, C_{0,2}) \}$ of $\tau$ in $ST(M)$. But Claim (\ref{C0=W}) then implies that $C_{0,1} = W = C_{0,2}$ for $W \in B_d$, contradicting the assumption that $\tau_1$ and $\tau_2$ are distinct.  Claim (\ref{Claimp>0}) follows. The claim implies, in particular, that if $\rank(M)=2$, then $B_d= \varnothing$. 

Now suppose $n=\rank(M) > 2$. Assume by induction that the statement of \autoref{ContractibleTits} holds for all $2 \leq n' <n$ or if $B_d$ is empty.  Assume $B_d$ is nonempty and fix $W \in B_d$. 
We let $\hat \tau$, $\hat \sigma$, and  $\hat \tau_i$, respectively, be obtained by intersecting $W$ with all of the submodules corresponding to vertices of $\tau$, $\sigma$, and $\tau_i$, after removing duplicates and removing any submodules equal to $0$ or $W$. We will see that $\hat \tau$, $\hat \sigma$, and  $\hat \tau_i$  are simplices of $T(W)$, ${\CB}(W)$, and $ST(W)$, respectively. Let $\hat \sigma_i = \hat \sigma \cup \hat \tau_i$.  We wish to apply the inductive hypothesis, and so 
we must check that all the hypotheses of \autoref{ContractibleTits} hold for $W$ and the simplices $\hat \tau$, $\hat \sigma$, $\hat \tau_i$, $\hat \sigma_i$. 

To do this, we make a third general claim: 
\begin{equation} \label{TauClaim} 
\begin{array}{r@{}l}
 \cdot \; & V_0 \cap W   =  0 \\ 
\cdot \;  &C_{0,i}  = W \text{ for all $i=1, \dots, r$} \\ 
\cdot \;  & W \supsetneq C_{1,i} \supsetneq C_{2,i} \supsetneq \dots \supsetneq C_{p,i} \supsetneq 0  \text{ for all $i$,}  \\
&  \text{ in particular, }  C_{\ell,i} \cap W = C_{\ell,i} \text{ for all $\ell$ and all $i$,}    \\ 
\cdot \; & (V_{\ell} \cap W) \oplus C_{\ell,i} = W  \text{ for all $i$ and $\ell$}   \\[1em]
   &   \text{Thus, }\\[1em] 
    \hat \tau & \; = \{V_1 \cap W,  \dots, V_p \cap W\} \\
\hat \tau_i & \; = \{(V_1 \cap W, \; C_{1,i}), \dots, (V_p \cap W, \; C_{p,i})\}. \\

\end{array}
\end{equation}

The first three items follow from the definition of $B_d$ and by Claim (\ref{C0=W}). The statement that $(V_{\ell} \cap W) \oplus C_{\ell,i} = W$ then follows from \autoref{CommonBasisIntersectionSum}. 
Claim (\ref{TauClaim}) follows.  

We now observe that the flag $\hat \tau$ in $W$ is nonempty. By Claim (\ref{Claimp>0}), $p>0$. Thus by Claim (\ref{TauClaim}) $$0 \; \;   \subsetneq \; \;   (V_p \cap W)\; \;    \subsetneq \; \;  W = (V_p \cap W) \oplus C_{p,i},$$ and so $V_p \cap W$ is a vertex of $\hat \tau$. We can therefore deduce from Claim (\ref{TauClaim}) that $ \hat \tau_i$ is indeed a simplex of $ST(W)$, and $\hat \tau_i$ projects to the flag $\hat \tau$ in $T(W)$ under the map $ST(W) \to T(W)$. 
Moreover, since $p>0$ by Claim (\ref{Claimp>0}) and $W \supsetneq C_{p,i} \supsetneq 0$ by Claim (\ref{TauClaim}), we further deduce that rank$(W) \geq 2$. 

We note that, since the common basis property is closed under taking intersection (\autoref{CommonBasisClosure}), the modules $\hat \sigma_i$ must have the common basis property. 

We next check that the simplices $\hat \tau_i$ are distinct. The assumption that the simplices $\tau_i$  of $ST(M)$ are distinct means that the chains $C_{0,i} \supsetneq C_{1,i} \supsetneq \dots \supsetneq C_{p,i}$ are pairwise distinct for $i=1, \dots, r$. But by Claim (\ref{C0=W}),  $W=C_{0,i}$ for all $i$, so $C_{1,i} \supsetneq \dots \supsetneq C_{p,i}$ must be pairwise distinct chains, and the simplices $\hat \tau_i$ of $ST(W)$ (as described in Claim (\ref{TauClaim}))  are distinct.  We conclude that all hypotheses of our inductive hypothesis hold for $W$, $\hat \tau$, $\hat \sigma$, $\hat \tau_i$, $\hat \sigma_i$.

To finish the proof, we will show that 
$$\Link^{<}_X(W) \cap G_{d-1} = \Link^{<}_X(W) =  T(W,\hat \sigma_1) \cap \dots \cap T(W,\hat \sigma_k).$$ 
Both $\Link^{<}_X(W)$ and  $T(W,\hat \sigma_1) \cap \dots \cap T(W,\hat \sigma_k)$ are full subcomplexes of $T(W)$, so it suffices to show that they have the same vertex set.  To do this, first let $U \in \Link^{<}_X(W)$, that is, $U$ is a summand of $W$ such that $\{U,W\}\cup \sigma_i$ are compatible for all $i$. Hence $U$ is compatible with $\hat \sigma_i$ for all $i$ since taking closure under intersection preserves compatibility (\autoref{CommonBasisClosure}). It follows that $$\Link^{<}_X(W) \subseteq  T(W,\hat \sigma_1) \cap \dots \cap T(W,\hat \sigma_k).$$
Now suppose $U \in T(W,\hat \sigma_1) \cap \dots \cap T(W,\hat \sigma_k)$, that is, $U$ is a summand of $W$ that is compatible with $\hat \sigma_i$ for all $i$. Fix $i$. We can find a basis of $W$ that is compatible with $\hat \sigma_i$ and $U$. Because $W$ is compatible with $\sigma_i$, we can also find a basis of a complement of $W$ that is compatible with $\sigma_i$. The union of these bases is a common basis of $\{U,W\} \cup \sigma_i$, proving that $U \in \Link^{<}_X(W)$. 


Since $ 2 \leq \rank(W) < \rank(M)=n$, we may assume the complex $$ \Link^{<}_X(W) \cap G_{d-1} =  T(W,\hat \sigma_1) \cap \dots \cap T(W,\hat \sigma_k)$$ is contractible by the inductive hypothesis on $n$. By \autoref{Morse}, this completes the proof. 
\end{proof}

We now prove the analogue of \autoref{ContractibleTits} for higher buildings. 

\begin{proposition} \label{higherContractible}
Let $M$ be a free $R$-module with $R$ a PID.  Let $\tau=\{V_0,\dots,V_p\}$ be a nonempty simplex of $T(M)$ and let $\sigma$ be a simplex of ${\CB}(M)$ containing $\tau$. Let $\tau_1,\dots,\tau_k$ be distinct simplicies of $ST(M,\sigma)$ in the preimage of $\tau$ under the map $ST(M,\sigma) \m T(M,\sigma)$.  Let $\sigma_i$ be the simplex of $\CB(M)$ obtained by adding the vertices of $\tau_i$ to $\sigma$. If $k \geq 2$, $a \geq 1$ and $\rank (M) \geq 2$, then $$T^{a,b}(M,\sigma_1) \cap \dots \cap T^{a,b}(M,\sigma_k) $$ is contractible.

\end{proposition}

\begin{proof}
We prove this by induction on $a+b$. The case that $a+b=1$ (so necessarily $a=1, b=0$) is \autoref{ContractibleTits}. Assume $a+b \geq 2$ and the statement holds for all $a'+b'<a+b$ with $a'\geq 1$. Let 
$$X=T^{a,b}(M,\sigma_1) \cap \dots \cap T^{a,b}(M,\sigma_k)$$
 and let $F_i(X)$ be the subcomplex of $X$ consisting of simplices: 
$$ (V_{0,1} < \dots< V_{n_1,1}  ) * \dots * (V_{0,a} < \dots< V_{n_a,a}  ) * (M_{0,1},\dots,M_{m_1,1}  ) * \dots * (M_{0,b},\dots,M_{m_b,b}  ) $$
 with $m_b \leq i$ if $b \geq 1$ or $n_a \leq i-1$ if $b=0$. In other words, we filter 
$$X \subseteq \underbrace{T(M) * \dots * T(M)}_{a\ \text{times}} * \underbrace{ST(M) * \dots * ST(M)}_{b\ \text{times}}$$ 
by the skeletal filtration of the last term in the join (the last $ST(M)$ factor, or the last $T(M)$ factor if $b=0$). The filtered piece $F_i(X)$ has simplices of dimension at most $i-1$ in the $(a+b)^{\mathrm{th}}$ join factor. In particular, when $i=0$, the last join factor must be the empty set, and
$$F_0(X) \cong  \left\{ \begin{array}{l}  T^{a,b-1}(M,\sigma_1) \cap \dots \cap T^{a,b-1}(M,\sigma_k), \text{ if $b> 0$} \\ T^{a-1,0}(M,\sigma_1) \cap \dots \cap T^{a-1,0}(M,\sigma_k), \text{ if $b= 0$}, \end{array}\right. $$ 
so $F_0(X)$ is contractible by induction on $a+b$. 

Fix $i>0$ and assume by induction on $i$ that $F_{i-1}(M)$ is contractible. Let $S_i$ denote the set of $(i-1)$-simplices of the last term of the join, i.e., the $(i-1)$-simplices of
$$   \left\{ \begin{array}{ll}   ST(M,\sigma_1) \cap \dots \cap ST(M,\sigma_k) , &\text{ if $b> 0$} \\ 
T(M,\sigma_1) \cap \dots \cap T(M,\sigma_k) , & \text{ if $b= 0$} . 
 \end{array}\right. $$ 
Apply \autoref{Morse2} to the complex $F_i(X)$, subcomplex $F_{i-1}(X)$, and set $S_i$. Then \begin{align*}F_i(X) &\simeq F_i(X)/F_{i-1}(X)   \qquad \qquad \text{since $F_{i-1}(X) $ is a contractible subcomplex,} 
\\& \simeq  \bigvee_{\rho \in S_i} \Sigma^{\dim(\rho)+1} \big( \Link_{F_i(X)}(\rho) \big)
\\& \simeq  \left\{ \begin{array}{ll}   \bigvee_{\rho \in S_i} \Sigma^{i}  \Big( T^{a,b-1}(M,\sigma_1 \cup \rho) \cap \dots \cap T^{a,b-1}(M,\sigma_k \cup \rho)\Big) &\text{ if $b> 0$} \\[1em] 
 \bigvee_{\rho \in S_i} \Sigma^{i} \Big( T^{a-1,0}(M,\sigma_1 \cup \rho) \cap \dots \cap T^{a-1,0}(M,\sigma_k \cup \rho) \Big)
& \text{ if $b= 0$} . 
 \end{array}\right. \end{align*} 
Here $\sigma_j \cup \rho$ denotes the simplex obtained by taking the union of the vertices of $\sigma_j$ and $\rho$. By induction on $a+b$, the complexes 
$$ T^{a,b-1}(M,\sigma_1 \cup \rho) \cap \dots \cap T^{a,b-1}(M,\sigma_k \cup \rho) \text{ and } T^{a-1,0}(M,\sigma_1 \cup \rho) \cap \dots \cap T^{a-1,0}(M,\sigma_k \cup \rho) $$ 
are contractible. Thus, by induction on $i$, the complex $F_i(X)$ is contractible  for all $i$. Since $F_{\rank(M)-1}(X)=X$, we conclude $X$ is contractible as claimed.  
\end{proof}

\subsection{Proof of the comparison theorem}

In this section, we prove that the homology of $T^{a,b}(M)$ only depends on the quantity $a+b$ if $a \geq 1$. We begin observing the following elementary topological lemma, which follows from the Mayer--Vietoris spectral sequence. 

\begin{lemma} \label{LemmaWedge}
Let $X$ be a $CW$ complex with $X=\cup_{\alpha \in S} X_\alpha$ for some subcomplexes $\{X_{\alpha}\}_{\alpha \in S}$.  Assume $\bigcap_{\alpha \in S} X_{\alpha}$ is not empty and let $p \in \cap_{\alpha \in S} X_{\alpha}$ be a 0-cell. View $X_\alpha$ as a based space with $p$ as the basepoint.  Assume all intersections of the form $$X_{\alpha_1} \cap \dots \cap X_{\alpha_k}$$ are acyclic provided $k \geq 2$ and each $\alpha_i$ is distinct. Then the map $$\bigvee_{\alpha \in S} X_\alpha \m X $$ is a homology equivalence.
\end{lemma}


We now prove the main result of this section which can be viewed as an unstable version of Waldhausen's additivity theorem.

\begin{theorem} \label{SplitvsNotSplit}
Let $R$ be a PID and $M$ a finite-rank free $R$-module. Let $\sigma$ be a (possibly empty) simplex of ${\CB}(M)$. Let $f\colon T^{a,b}(M, \sigma) \m T^{a+b,0}(M, \sigma)$ denote the map induced by forgetting complements. 
For $a\geq 1$, the map $f$ is a homology equivalence. 
\end{theorem}

The map $f$ is likely a homotopy equivalence, but  we only need and prove this weaker statement. 


\begin{proof}[Proof of \autoref{SplitvsNotSplit}]
We will prove this theorem by induction on $a+b$. The statement is immediate if $b=0$ (in particular if $a+b \leq 1$) so assume otherwise.  Let $F_i(T^{a,b}(M,\sigma) )$ and $F_i(T^{a+b,0}(M,\sigma) )$, respectively, be the filtrations induced by the skeletal filtration on the final join factor $ST(M)$ and $T(M)$, respectively, as in the proof of \autoref{higherContractible}. Explicitly, $F_i(T^{a,b}(M,\sigma) )$ is the subcomplex of $T^{a,b}(M,\sigma)$ consisting of simplices
$$ (V_{0,1} < \dots< V_{n_1,1}  ) * \dots * (V_{0,a} < \dots< V_{n_a,a}  ) * (M_{0,1},\dots,M_{m_1,1}  ) * \dots * (M_{0,b},\dots,M_{m_b,b}  ) $$
with $m_b \leq i$. The filtration $F_i(T^{a+b,0}(M,\sigma) )$ is defined analogously with $n_{a+b} \leq i-1$. 

Observe that $F_0(T^{a,b}(M,\sigma) )$ is isomorphic to $T^{a, b-1}(M,\sigma)$ if $b\geq 1$ and $T^{a-1, b}(M,\sigma) $ if $b=0$. Hence $F_0(T^{a,b}(M,\sigma) ) \to F_0(T^{a+b,0}(M,\sigma) )$ is a homology equivalence by induction on $a+b$.

Suppose $i>0$. Let $S_i^{a,b}$ denote the set of $(i-1)$-simplices of $ST(M,\sigma)$ and let $S_i^{a+b,0}$ denote the set of $(i-1)$-simplices of $T(M,\sigma)$. By \autoref{Morse2} and as in the previous proof,
$$F_i( T^{a,b}(M,\sigma)  )/ F_{i-1}( T^{a,b}(M,\sigma)  ) \simeq \bigvee_{\rho \in S_i^{a,b}} \Sigma^{i} \Big( T^{a,b-1}(M,\rho \cup \sigma) \Big)  \text{ and}$$
$$ F_i( T^{a+b,0}(M,\sigma)  )/ F_{i-1}( T^{a+b,0}(M,\sigma) \simeq \bigvee_{\rho \in S_i^{a+b,0}} \Sigma^{i} \Big( T^{a+b-1,0}(M,\rho \cup \sigma)\Big).$$ Note the distinct index sets of the wedges. 

We observe that $f$ preserves the filtrations above. We will prove that $f\colon T^{a,b}(M, \sigma) \m T^{a+b,0}(M, \sigma)$ is a homology equivalence by checking that $f$ induces a homology equivalence on the associated graded of this filtration.  To show $$ F_i( T^{a,b}(M,\sigma)  )/ F_{i-1}( T^{a,b}(M,\sigma)  ) \m F_i( T^{a+b,0}(M,\sigma)  )/ F_{i-1}( T^{a+b,0}(M,\sigma))$$ is a homology equivalence, it suffices to show  $$ f_{\tau}\colon  \bigvee_{\rho \in f^{-1}(\tau)} T^{a,b-1}(M,\rho \cup \sigma) \m T^{a+b-1,0}(M,\tau \cup \sigma) $$ is a homology equivalence for all $\tau \in S_i^{a+b,0}$. By our induction hypothesis, it suffices to show $$ \hat{f}_{\tau}\colon  \bigvee_{\rho \in f^{-1}(\tau)} T^{a+b-1,0}(M,\rho \cup \sigma) \m T^{a+b-1,0}(M,\tau \cup \sigma) $$ is a homology equivalence. We will show that $\hat{f}_{\tau}$ is an equivalence using \autoref{LemmaWedge} applied to a cover of $T^{a+b-1,0}(M,\tau \cup \sigma)$. We cover $T^{a+b-1,0}(M,\tau \cup \sigma)$ by subspaces of the form $T^{a+b-1,0}(M,\rho \cup \sigma) $ for $\rho \in f^{-1}(\tau)$. This is a cover because whenever a collection of submodules are compatible with a flag $\tau$, then the collection is compatible with some splitting $\rho \in f^{-1}(\tau)$. In other words, we can choose complements to the terms in the flag $\tau$ compatible with the collection of submodules.

We wish to apply \autoref{higherContractible}; observe that we may assume that rank$(M) \geq 2$ otherwise the map $f\colon T^{a,b}(M, \sigma) \m T^{a+b,0}(M, \sigma)$  has empty domain and codomain and hence is a homology equivalence. \autoref{higherContractible} implies that finite intersections of two or more subcomplexes in the cover are contractible. To see that $$ \bigcap_{\rho \in f^{-1}(\tau)}  T^{a+b-1,0}(M,\rho \cup \sigma)   $$ is not empty, observe that $\tau$ is a simplex of the intersection. \autoref{LemmaWedge} now implies that $\hat{f}_{\tau}$ and hence $f$ is a homology equivalence. 
\end{proof}

\section{Algebraic properties of higher Tits buildings} \label{Section-AlgebraicPropertiesHigherTitsBuildings}

In this section, we describe spaces $D^{a,b}(M)$ which are homotopy equivalent to iterated suspensions of the complexes $T^{a,b}(M)$ studied in the previous section. The reason for considering two different models of higher Tits buildings is that the complexes $T^{a,b}(M)$ are smaller and hence more convenient for combinatorial Morse theory while the complexes $D^{a,b}(M)$ are larger and have better algebraic properties. We begin with categorical preliminaries.

\subsection{Categorical framework} \label{CategoricalSubsection}
 
In this subsection, we recall basic facts about Day convolution and bar constructions. Throughout, we will fix a ring $R$, which we sometimes suppress from the notation. We remark that no constructions or results in \autoref{CategoricalSubsection} require $R$ to be a PID. 

\begin{definition}

Let $\VB$ denote the groupoid with objects finite-rank free $R$-modules and morphisms all $R$-linear isomorphisms.  A \emph{$\VB$-module} is a functor from $\VB$ to the category $\Ab$ of $\Z$-modules. Let $\Mod_{\VB}$ denote the category of $\VB$-modules. Similarly define $\VB$-sets, $\VB$-based sets, $\VB$-chain complexes, and $\VB$-based spaces.
\end{definition}

The notation $\VB$ is used in the field of representation stability to connote ``vector spaces" and ``bijective maps" \cite{PutmanSam}. We sometimes use the notation $\VB(R)$ to stress the dependence of the groupoid on $R$.  The groupoid $\VB(R)$ is equivalent to the groupoid $\GL(R)$ defined in the introduction, but this framework will conceptually simplify some arguments and clarify definitions like \autoref{Day}. 

Given a functor $V$ from $\VB$ to some category and a finite-rank free $R$-module $M$, let $V(M)$ denote the value of $V$ on $M$. Let $V_n$ denote $V(R^n)$. Let $\cC=\Ab$, ${\Set_*}$, $\Ch$, or ${\Top_*}$.  The category of functors $\Fun(\VB,\cC)$ has a symmetric monoidal product. It is an instance of a construction called \emph{Day convolution} or induction product.

\begin{definition} \label{Day}
 Let $\cC= \Ab$, ${\Set_*}$, $\Ch$, or ${\Top_*}$ and let $V,W \in \Fun(\VB,\cC)$. Let $\otimes_{\cC}$ denote tensor product if $\cC$ is $\Ab$ or $\Ch$, and smash product if $\cC$  is ${\Set_*}$ or ${\Top_*}$.   Let $V \otimes_{\VB} W \in \Fun(\VB,\cC)$ be defined by  
 $$(V \otimes_{\VB} W)(M):= \coprod_{A \oplus B=M} V(A) \otimes_{\cC} W(B)$$ on objects. An $R$-linear isomorphism $f \colon M \to M'$ induces the map 
 \begin{align*}
 \coprod_{A \oplus B=M} V(A) \otimes_{\cC} W(B) \longrightarrow \coprod_{A' \oplus B'=M'} V(A') \otimes_{\cC} W(B')
 \end{align*}
 that maps the factor indexed by the decomposition $A \oplus B=M$ to the factor indexed by $f(A) \oplus f(B) = M'$ and maps $V(A) \otimes_{\cC} W(B)$ to $V(f(A)) \otimes_{\cC} W(f(B))$ via functoriality.  
\end{definition}
In the indexing set of the coproduct, $A, B$ are submodules of $M$, and the sum $A \oplus B$ is their internal direct sum.  The symmetry $(V \otimes_{\VB} W)(M) \to (W \otimes_{\VB} V)(M)$ is defined as the map 
$$ \coprod_{A \oplus B=M} V(A) \otimes_{\cC} W(B) \longrightarrow \coprod_{A \oplus B=M} W(A) \otimes_{\cC} V(B)$$ 
 induced by the symmetry  $V(A) \otimes_{\cC} W(B) \to W(B) \otimes_{\cC} V(A)$ from the summand $A \oplus B = M$ to the summand $B \oplus A = M$. Associativity is also induced by associativity of $\otimes_{\cC}$. 

In the case of $\VB$-modules,  $$(V \otimes_{\VB} W)_n \cong \bigoplus_{a +b=n} \Ind^{\GL_n(R)}_{\GL_a(R) \times \GL_b(R)}V_a \otimes_{\cC} W_b.$$ The product $\otimes_{\VB}$ gives a symmetric monodial structure on $\Fun(\VB,\cC)$ with unit the object defined in \autoref{NotationUnit} (see, for example, \cite[Section 2.1]{MNP}). This product is an instance of Day convolution, which in general endows a symmetric monoidal structure on a functor category from a (essentially) small symmetric monoidal category to a symmetric monoidal category; see (for example) \cite[Section 2]{MayZhangZou}.

\begin{notation} \label{NotationUnit} Let $\cC= \Ab$, ${\Set_*}$, $\Ch$, or ${\Top_*}$. Given an object $X$ of $\mathcal{C}$, we may view $X$ as a functor in $\Fun(\VB, \mathcal{C})$ that takes the value $X$ on modules of rank $0$ and takes the zero object of $\mathcal{C}$ elsewhere. With this convention, the unit of $(\mathcal{C}, \otimes_{\mathcal{C}})$ defines a unit in $(\Fun(\VB,\cC), \otimes_{\VB})$. 
 
 We denote this unit generically by ${\bf 1}$. We may write $\Z$ for the unit when $\cC= \Ab$ or $\cC=\Ch$. We may write $S^0$ for the unit when $\cC={\Set_*}$ or $\cC={\Top_*}$. 
 \end{notation}


 

This symmetric monoidal structure  $(\Fun(\VB,\cC), \otimes_{\VB}, {\bf 1})$  allows us to make sense of (unital) monoid objects, right/left-modules over monoids, etc. Note that $\mathbf 1$ admits a unique structure as a monoid object. If $\cC$ is $\Ab$ or $\Ch$, we often call monoid objects \emph{$\VB$-rings}.

\begin{definition}
Let $A$ be a monoid in $\Mod_{\VB}$ with $M$ a right $A$-module and $N$ a left $A$-module. Let $M \otimes_A N$ be the coequalizer of the two natural maps $$ M \otimes_{\VB} A \otimes_{\VB} N \rightrightarrows M \otimes_{\VB} N.$$
\end{definition}

The functor  $\cdot \otimes_A N$ is right exact, and the category of $A$-modules has enough projectives (see, for example, \cite[Proof of Theorem 5.9]{MNP}). We can therefore define its left-derived functors. 

\begin{definition}
Let $A$ be a monoid in $\Mod_{\VB}$ with $M$ a right $A$-module and $N$ a left $A$-module.
 Let $\Tor^i_A(\;\cdot\;,N)$ denote the $i$th left-derived functor of $\cdot \otimes_A N$.
\end{definition}

These $\Tor$ groups can be used to formulate the notion of Koszul rings.

\begin{definition} \label{DefnAugmented}

Let $\cC= \Ab$, ${\Set_*}$, $\Ch$, or ${\Top_*}$. We say a monoid $A$ in $\Fun(\VB, \cC)$ is \emph{augmented} if it is equipped with a map of monoids $A \m {\bf 1}$.  An augmented monoid $A$ is \emph{connected} if this map is an isomorphism in $\VB$-degree $0$. 
\end{definition}

\begin{definition} \label{DefnKoszul}
An augmented $\VB$-ring $A$ in $\Mod_{\VB}$ is called \emph{Koszul} if $$\Tor^A_i(\Z,\Z)_n \cong 0 \qquad \text{ for $n \neq i$.} $$   In this case, we define the \emph{Koszul dual} of $A$ to be the $\VB$-module with value $\Tor^A_n(\Z,\Z)_M$ on a rank-$n$ free $R$-module $M$. 
\end{definition}

As in classical algebra, $\Tor$ groups can be computed via bar constructions.

\begin{definition}
Let $\cC= \Ab$, ${\Set_*}$, $\Ch$, or ${\Top_*}$. Let $A$ be a monoid in $\Fun(\VB,\cC)$, $M$ a right $A$-module and $N$ a left $A$-module. Let $B_\bullet(M,A,N)$ be the simplicial object in $\Fun(\VB,\cC)$ with $p$-simplices $$B_p(M,A,N):=M \otimes_{\VB} A^{\otimes_{\VB}p} \otimes_{\VB} N,$$ face maps induced by the monoid and module multiplication maps \begin{align*} M \otimes_{\VB} A &\m M, \\  A \otimes_{\VB} A& \m A,\\ A \otimes_{\VB} N &\m N, \end{align*} and degeneracies induced by the unit of $A$. 

For $\cC=\Ch$, let $B(M,A,N) \in \Fun(\VB,\Ch)$ denote the total complex of the double complex of normalized chains associated to the simplicial object  $B_\bullet(M,A,N)$. For $\cC=\Ab$, we make the same definition by viewing the abelian groups as chain complexes concentrated in homological degree 0. When $\cC$ is ${\Set_*}$  or ${\Top_*}$, let $B(M,A,N) \in \Fun(\VB,\Top_*)$ denote the thin geometric realization of $B_\bullet(M,A,N)$. We define the (thin or thick) geometric realization of a based object as the realization of the associated unbased object. The thin geometric realization will be based, since the inclusion of the basepoints induces a map from the constant simplicial set on a point; its thin realization is a point. 


\end{definition}

The following result appears in Miller--Patzt--Petersen \cite[Proposition 2.33]{MillerPatztPetersen}. 

\begin{proposition} \label{PropositionTorViaBarConstruction} Let $\cC=\Ab$. Then $$H_i(B(M,A,N)) \cong \Tor_i^A(M,N)$$ if either $M_n$ is a free abelian group for all $n$ or $N_n$ is a free abelian group for all $n$. 
\end{proposition}

\begin{definition}

 Let $\cC=\Ab$, ${\Set_*}$,  $\Ch$, or ${\Top_*}$.  Let $A$ be an augmented monoid in $\Fun(\VB,\cC)$. Let $BA_{\bullet} :=B_{\bullet}({\bf 1},A,{\bf 1})$ and $BA :=B({\bf 1},A,{\bf 1})$.
\end{definition} 

The following is May  \cite[Definition 11.2]{MayGeometryIteratedLoopSpaces}.

\begin{definition}
Let $X_\bullet$ be a simplicial space. The \emph{$p$th latching object} $sX_p \subseteq X_p$ is the union of the images of the degeneracy maps $s_i:X_{p-1} \to X_p$. The simplicial space $X_\bullet$ is called \emph{proper} if $(X_p,sX_p)$ is a strong NDR-pair.

\end{definition}

\begin{lemma} \label{Bconn}
Let $A$ be a connected augmented monoid object in  $\VB$-based spaces. Let $\beta \geq -2$. If $A(R^n)$  is $(\alpha n+\beta)$-connected for all $n$ and $BA_{\bullet}(R^n)$ is a  proper simplicial space for all $n$  then $BA(R^n)$ is $(\alpha n+\beta+1)$--connected for all $n$.
\end{lemma}

\begin{proof} Observe $BA_{p}(R^n)$ is 
$$ \bigvee_{n_1 + n_2+ \dots + n_p=n} \; \; \bigwedge_{i=1, \dots, p} A(R^{n_i}).$$
When $p=0$, this is contractible, and hence $(\alpha n+\beta+1)$-connected. In general, $BA_{p}(R^n)$  is  $(\alpha n+p\beta+p-1)$-connected, and $(\alpha n+p\beta+p-1) \geq (\alpha n+\beta+1-p)$ for $p>0$.

Since $BA_{\bullet}(R^n)$ is assumed to be proper, its thick and thin realizations are homotopy equivalent (see, for example, tom Dieck \cite[Proposition 1]{Dieck}). The result now follows from Ebert--Randal-Williams \cite[Lemma 2.4]{EbertRW}. See also May \cite[Theorem 11.12]{MayGeometryIteratedLoopSpaces}.
\end{proof}

If $A$ is a commutative monoid object, then so is $BA$. Here it is essential that we use thin geometric realizations instead of thick geometric realizations; otherwise, these spaces will not have units. This result is a generalization of the classical fact that the bar construction of an abelian monoid is again a monoid; see for example May \cite[Corollary 11.7]{MayGeometryIteratedLoopSpaces}, 
\cite[Remarks 8.10]{MayClassifyingSpaces}.  Thus, we can iterate the bar construction.

\subsection{Simplicial models of higher buildings}

In this section, we introduce spaces $D^{a,b}(M)$, and prove they are homotopy equivalent to iterated suspensions of the complexes $T^{a,b}(M)$.
\begin{definition}
Let $L_\bullet \colon \VB \to \mathrm{Fun}(\Delta^{op}, \mathrm{Set})$ be the $\VB$-simplicial-set whose value on a finite-rank free $R$-module $M$ is the simplicial set whose $p$-simplices $L_p(M)$ is the set of flags $V_0 \subseteq \dots \subseteq V_p $ of (not necessarily distinct, not necessarily proper, not necessarily nonzero) summands of $M$. The degeneracy map $s_i\colon L_p(M) \m L_{p+1}(M)$ is the map that duplicates the $i$th summand. The face map $d_i\colon L_p(M) \m L_{p-1}(M)$ is the map that forgets the $i$th summand.  An $R$-linear isomorphism acts in the obvious way. Let $\mathring{L}_\bullet$ denote the $\VB$-sub-simplicial-set where $\mathring{L}_p(M)$ is the set of flags $V_0 \subseteq \dots \subseteq V_p $ with $V_0 \neq 0$ or $V_p \neq M$.  
\end{definition}

It is straightforward to verify that the face and degeneracy maps commute with the maps induced by morphisms in $\VB$. 

We use the notation $L_\bullet$ to connote ``lattices"  (in the sense of Rognes \cite{Rog1}), and $\mathring{L}_\bullet$  are referred to as ``not full lattices". We now define their split versions. 

\begin{definition}
Let $SL_\bullet\colon \VB \to \mathrm{Fun}(\Delta^{op}, \mathrm{Set})$ be the $\VB$-simplicial-set, which, when evaluated on a free $R$-module $M$, has $p$-simplices $SL_p(M)$ given by splittings $(M_0 , \dots , M_{p+1})$ of $M$ into (not necessarily proper or nonzero) summands. Define the degeneracy map $s_i\colon SL_p(M) \m SL_{p+1}(M)$ to be the map that inserts the zero summand between $M_i$ and $M_{i+1}$. Define the face map $d_i\colon SL_p(M) \m SL_{p-1}(M)$ to be the map that replaces $M_i$ and $M_{i+1}$ with $M_i \oplus M_{i+1}$. Let $S\mathring{L}_\bullet$ denote the $\VB$-sub-simplicial-set where $S\mathring{L}_p(M)$ is the set of splittings $(M_0, \dots,  M_{p+1} )$ with $M_0 \neq 0$ or $M_{p+1} \neq 0$.

An $R$-linear isomorphism $f\colon M \to M'$ induces a map $SL_p(M) \to SL_p(M')$ mapping 
$$ (M_0 , \dots , M_{p+1}) \longmapsto (f(M_0) , \dots , f(M_{p+1})).$$ 
\end{definition}

\begin{definition}
Let $L^{a,b}_{\bullet,\dots,\bullet}$ be the $(a+b)$-fold $\VB$-simplicial-set whose value on a free $R$-module $M$ and a tuple $(p_1,\dots,p_{a+b})$ is the subset of 
$$L_{p_1}(M) \times \dots \times L_{p_a}(M) \times SL_{p_{a+1}}(M) \times \dots \times SL_{p_{a+b}}(M)$$ of summands that together have a common basis. The $\VB$-multi-simplicial structure is induced from the  $\VB$-simplicial structures on $L_\bullet$ and $SL_{\bullet}$. 
Let $\mathring{L}^{a,b}_{\bullet,\dots,\bullet}$ be the $(a+b)$-fold $\VB$-subsimplicial-set of $L^{a,b}_{\bullet,\dots,\bullet}$ whose $(n_1, \dots, n_{a+b})$-simplices are contained in the $(n_1, \dots, n_{a+b})$-simplices of 
$$L_{n_1}(M) \times \dots \times L_{n_{i-1}}(M)  \times \mathring{L}_{n_i}(M)  \times  L_{n_{i+1}}(M)  \times  \dots   \times L_{n_a}(M)  \times SL_{n_{a+1}}(M)  \times \dots \times SL_{n_{a+b}}(M) $$ for some $1 \leq i \leq a$ or  $$L_{n_1}(M)  \times \dots \times   L_{n_a}(M)  \times SL_{n_{a+1}}(M)  \times \dots  \times SL_{n_{i-1}}(M)  \times S\mathring{L}_{n_i}(M)  \times  SL_{n_{i+1}}(M)  \times \dots \times SL_{n_{a+b}}(M) $$ for some $(a+1) \leq i \leq (a+b)$. 
\end{definition}

\begin{definition}
Let $D^{a,b}_{\bullet,\dots,\bullet}$ be the $(a+b)$-fold $\VB$-simplicial-based-set defined by the quotients

$$ D^{a,b}_{p_1, \ldots, p_{a+b}}(M) = L^{a,b}_{p_1, \ldots, p_{a+b}}(M) / \mathring{L}^{a,b}_{p_1, \ldots, p_{a+b}}(M)$$ 
with basepoint given by the image of $\mathring{L}^{a,b}_{p_1, \ldots, p_{a+b}}(M)$ and $\VB$-multisimplicial-structure induced by $L^{a,b}_{\bullet, \dots, \bullet}$. 
\end{definition}

Given a (based) multi-simplicial space $X_{\bullet,\dots,\bullet}$, we let $X_\bullet$ denote the diagonal and let $X$ denote the thin geometric realization. 

\begin{definition}  \label{Dab}
Let $$\mu_{n_1,\dots,n_{a+b}}\colon  D^{a,b}_{n_1,\dots,n_{a+b}}(M) \wedge D^{a,b}_{n_1,\dots,n_{a+b}}(N) \longrightarrow D^{a,b}_{n_1,\dots,n_{a+b}}(M \oplus N)$$ be the map defined factorwise  by the formula 
$$(V_0 \subseteq \dots \subseteq V_p) \times (W_0 \subseteq \dots \subseteq W_p) \longmapsto (V_0 \oplus W_0 \subseteq \dots \subseteq V_p \oplus W_p)$$ on $L_\bullet$ factors and defined by $$(M_0, \dots, M_p) \times (N_0, \dots, N_p) \longmapsto (M_0 \oplus N_0, \dots, M_p \oplus N_p)$$ on $SL_\bullet$ factors.
\end{definition}

Then $D^{a,0}_{\bullet}$ (we referred to its realization as $D^a$ in the introduction) is Rognes' higher building \cite[Definition 3.9]{Rog1}, and  $D^{0,b}_{\bullet, \dots, \bullet}$ is Galatius--Kupers--Randal-Williams' split version \cite[Definition 5.9]{e2cellsIV}. The fact that these constructions agree with those of Rognes and Galatius--Kupers--Randal-Williams is not immediate but follows quickly from Galatius--Kupers--Randal-Williams \cite[Lemma 5.6 and Lemma 5.7]{e2cellsIV}. In \cite[Remark 5.2]{e2cellsIV}, Galatius--Kupers--Randal-Williams explain that their notions of $k$-dimensional buildings and stable buildings are isomorphic to Rognes’; compare \cite[Definition 5.4]{e2cellsIV} with \cite[Definition 3.9]{Rog1} and  \cite[Definition 5.5]{e2cellsIV} with \cite[Definition 10.8]{Rog1}.

\begin{lemma}
The maps $\mu_{n_1,\dots,n_{a+b}} $ assemble to form a simplicial map:  $$\mu_{\bullet,\dots,\bullet}\colon  D^{a,b}_{\bullet,\dots,\bullet}(M) \wedge D^{a,b}_{\bullet,\dots,\bullet}(N) \m D^{a,b}_{\bullet,\dots,\bullet}(M \oplus N).$$
\end{lemma}


Observe that $D^{a,b}_{0,\dots,0}(M)$ is  the basepoint for all $M\neq 0$. By convention, $D^{a,b}_{0,\dots,0}(0)$ is $S^0$ since $L^{a,b}_{0,\dots,0}(0)$ is a point and $\mathring{L}^{a,b}_{0,\dots,0}(0)$ is empty; similarly for the split versions of these complexes.
In particular, there is a natural isomorphism $\iota \colon  {\bf 1} \m  D^{a,b}_{0,\dots,0}$. 

\begin{definition}
Let ${\bf 1}_{\bullet,\dots,\bullet} $ be the constant $(a+b)$-fold simplicial  $\VB$-based set on ${\bf 1}$.  Let $\iota_{\bullet,\dots,\bullet} \colon {\bf 1}_{\bullet,\dots,\bullet}  \m  D^{a,b}_{\bullet,\dots,\bullet}$ be the unique extension of the natural map $\iota \colon  {\bf 1} \m  D^{a,b}_{0,\dots,0}$. 
\end{definition}

Verifying the following two lemmas is elementary but involved. Compare \autoref{GKRW6.6} to Galatius--Kupers--Randal-Williams \cite[Lemma 6.6]{e2cellsIV}. \autoref{GKRW6.6} follows from associativity and commutativity of direct sums. \autoref{Bsplit} follows from the fact that sum and intersection operation are well-behaved for modules with the common basis property; see (for example) \autoref{CommonBasisIntersectionSum}. 

\begin{lemma} \label{GKRW6.6}
The maps $\mu_{\bullet,\dots,\bullet}$ and $\iota_{\bullet,\dots,\bullet} $ give $D^{a,b}_{\bullet,\dots,\bullet}$ the structure of a commutative monoid object in the category of $(a+b)$-fold simplicial $\VB$-based sets.
\end{lemma}

\begin{lemma} \label{Bsplit}
There is a natural isomorphism of $(a+b+1)$-fold simplicial $\VB$-based sets $B_\bullet D^{a,b}_{\bullet,\dots,\bullet} \cong D^{a,b+1}_{\bullet,\dots,\bullet}$.
\end{lemma}

%


\subsection{Comparing models of higher buildings}

We now show that $D^{a,b}(M)$ is homotopy equivalent to an iterated suspension of $T^{a,b}(M)$. This fact appeared in Galatius--Kupers--Randal-Williams \cite[Lemma 6.1]{e2cellsIV} in the case $(a,b)$ is $(1,0)$, and the cases when $a=0$ or $b=0$ are implicit in their work.

\begin{lemma} \label{suspend}
$D^{a,b}(M) \simeq \Sigma^{a+b+1} T^{a,b}(M) $.
\end{lemma}

\begin{proof}
Note that the diagonal $L^{a,b}_{\bullet}(M)$ admits two extra degeneracies. These are induced by the extra degeneracies of $L_\bullet(M)$ given by $$(V_0 \subseteq \dots \subseteq V_p ) \mapsto (0 \subseteq V_0 \subseteq \dots \subseteq V_p )   $$ and $$(V_0 \subseteq \dots \subseteq V_p ) \mapsto ( V_0 \subseteq \dots\subseteq V_p \subseteq M)   $$  and the extra degeneracies of $SL_\bullet(M)$ given by $$(M_0  , \dots , M_p ) \mapsto (0 , M_0 , \dots , M_p )   $$ and $$(M_0  , \dots , M_p ) \mapsto ( M_0 , \dots , M_p,0 ).   $$ Since $L^{a,b}(M)$ has an extra degeneracy, it is contractible (see e.g. Goerss--Jardine \cite[Lemma 5.1]{GoerssJardine}). Since $\mathring{L}^{a,b}(M) \m L^{a,b}(M)$ is a cofibration and $D^{a,b}(M)= L^{a,b}(M)/\mathring{L}^{a,b}(M)$, we conclude that $D^{a,b}(M) \simeq \Sigma (\mathring{L}^{a,b}(M))$.


There are inclusions of spaces $T(M) \hookrightarrow \mathring{L}(M)$ and $ST(M) \hookrightarrow S\mathring{L}(M)$. 
For $0 \leq k \leq (a+b)$, let $F_k$ be the subcomplex of $\mathring{L}^{a,b}(M)$ where the first $k$ $L(M)$ and $SL(M)$ factors are in $T(M)$ or $ST(M)$, i.e., 
for each $i\leq k$ the $i$th factor has the form 
$$(V_0 \subseteq \dots \subseteq V_p)  \text{ with }  V_0 \neq 0 \text{ and } V_p \neq M \text{ (but inclusions may not be strict)}$$
or 
$$(M_0,\dots,M_p) \text{ with }  M_0, M_p \neq 0  \text{ (but other terms $M_i$ may be 0)}.$$

 We will show that $F_{k-1}$ is homotopy equivalent to the suspension of $F_k$ by defining contractible subspaces $N_k$ and $S_k$ of $F_{k-1}$ that we view as the `northern' and `southern' hemispheres, which intersect in the `equator' $F_k$. Let $N_k$ be subcomplex of $F_{k-1}$ where the $k$th factor is of the form $$(V_0 \subseteq \dots \subseteq V_p) \text{ with } V_p \neq M $$ with $k \leq a$, and is of the form $$(M_0,\dots,M_p) \text{ with } M_p \neq 0$$ if $k>a$.  Let $S_k$ be subcomplex of $F_{k-1}$ where the $k$th factor is of the form $$(V_0 \subseteq \dots \subseteq V_p)  \text{ with } V_0 \neq 0 $$ if $k \leq a$, and is of the form $$(M_0,\dots,M_p) \text{ with } M_0 \neq 0$$ if $k>a$. Then $F_{a+b} \cong T^{a,b}(M)$, $F_0 \cong \mathring{L}^{a,b}(M)$,   $S_k \cap N_k \cong F_k$, and $S_k \cup N_k \cong F_{k-1}$. Note that $S_k$ and $N_k$ are contractible since they have cone points. Thus $F_k \simeq \Sigma F_{k+1}$. This implies $\mathring{L}^{a,b}(M) \simeq \Sigma^k T^{a,b}(M)$. Since $D^{a,b}(M) \simeq \Sigma (\mathring{L}^{a,b}(M))$, the claim follows. 
\end{proof}

Since the suspension of a homology equivalence gives a homotopy equivalence, \autoref{SplitvsNotSplit} and \autoref{suspend} imply the following. 

\begin{corollary} \label{SplitvsNotSplitD}
Let $a \geq 1$. Then the `forget the complement' map $D^{a,b}(M) \m D^{a+b,0}(M)$ is a homotopy equivalence. 
\end{corollary}

We now prove \autoref{mainLemma} which states that $B D^{k,0}(M) \simeq D^{k+1,0}(M)$.

\begin{proof}[Proof of \autoref{mainLemma}]  By \autoref{Bsplit}, there is a natural isomorphism of $(k+1)$-fold simplicial $\VB$-based sets $B_\bullet D^{k,0}_{\bullet,\dots,\bullet} \cong D^{k,1}_{\bullet,\dots,\bullet}$. \autoref{mainLemma} then follows from \autoref{SplitvsNotSplitD}. 
\end{proof}

\section{Rognes' connectivity conjecture} \label{SectionRognesConjecture}

In this section, we prove Rognes' connectivity conjecture in the case of fields, \autoref{ConnectivityThm}.

\subsection{Comparing $T^{k,0}_n(R)$ and ${\CB}_n(R)$}

In this subsection, we describe maps between $T^{k,0}_n(R)$ for different values of $k$ and also a map to ${\CB}_n(R)$.

\begin{definition} \label{DefnInducedMap}
Given an injection $f\colon \{1,\dots,k\} \m \{1,\dots,j\}$, let $f \colon  T_n^{k,0}(R) \m T_n^{j,0}(R)$ denote the simplicial embedding defined on vertices by $$\underbrace{\emptyset * \dots * \emptyset}_{i-1} * V * \emptyset* \dots * \emptyset \longmapsto \underbrace{\emptyset * \dots * \emptyset}_{f(i)-1} * V * \emptyset* \dots * \emptyset.$$
That is, $f$ maps a vertex in the $i$th join factor to a vertex in the $f(i)$th join factor.  Let $\iota \colon T_n^{k,0}(R) \m T_n^{k+1,0}(R) $ be the map induced by the standard inclusion $\{1,\dots,k\} \hookrightarrow \{1, \dots, k, k+1\}$. 
\end{definition}

Consider the colimit $ \colim_k T_n^{k,0}(R)$ defined by the simplicial embeddings $\iota$. Concretely, $\colim_k T_n^{k,0}(R)$ is isomorphic to the simplicial complex $\bigcup_k T_n^{k,0}(R)$. This is the subsimplicial complex of the infinite join 
$$T_n(R) * T_n(R) * \dots$$
where a simplex is the join of (finitely many) flags whose union has the common basis property. Since the maps $\iota$ are inclusions of simplicial complexes, this colimit agrees with the homotopy colimit.

\begin{definition}
Let $\pi \colon  T_n^{k,0}(R) \m {\CB}_n(R)$ be the simplicial map defined on vertices by 
 $$ \emptyset * \dots * \emptyset * V * \emptyset* \dots * \emptyset \longmapsto V$$   
 and defined on simplices by 
$$*_i \sigma_i  \; \longmapsto \;  \bigcup_i \sigma_i.$$   
\end{definition}

Note that $\pi$ commutes with $\iota$ so we obtain a map $$ \pi\colon    \colim_k T_n^{k,0}(R) \m {\CB}_n(R).$$ 


The following is a desuspension of Rognes' result \cite[Lemma 14.6]{Rog1}.

\begin{proposition} \label{colimTD}
The map $$ \pi\colon  \colim_k T_n^{k,0}(R) \m {\CB}_n(R) $$ is a homotopy equivalence.

\end{proposition}

\begin{proof}
Let  $\sigma=[V_0,\dots,V_p]$ in ${\CB}_n(R)$ be a simplex of ${\CB}_n(R)$. It suffices to show $\pi^{-1}(\sigma)$ is weakly contractible; see Quillen \cite[Theorem A and the Example that follows]{QuillenKTheory1}. Let $g\colon S^d \m \pi^{-1}(\sigma)$ be a simplicial map from some simplicial structure on the $d$-sphere $S^d$. Since $S^d$ is compact, the image of $g$ must be contained in the image of $\iota\colon  T^{m,0}_n(R) \m \colim_k T_n^{k,0}(R)$ for some $m$. The map $g$ has image in the star in $\pi^{-1}(\sigma)$ of $$\underbrace{\emptyset * \dots * \emptyset}_{m} * V_0 * \emptyset* \dots $$ and hence $g$ is nullhomotopic.
\end{proof}

\subsection{The fundamental group of $T^{k,0}_n(F)$ and ${\CB}_n(F)$}

Let $F$ be a field. Our goal is to show $T^{k,0}_n(F)$ and ${\CB}_n(F)$ are simply connected if $k \geq 2$ and $n \geq 3$. The following is implicit in Galatius--Kupers--Randal-Williams \cite{e2cellsIV}.

\begin{lemma} \label{lemJoin}
For $F$ a field, the inclusion $T^{2,0}_n(F) \m T_n(F) * T_n(F)$ is an isomorphism. 
\end{lemma}

\begin{proof}
This is the classical result that the union of any pair of flags admits a common basis. See (for example) Abramenko--Brown \cite[Section 4.3]{AbramenkoBrown}, where this result is proved in order to establish that the Tits building is a combinatorial building.  The result could also be proved inductively using our \autoref{BigFlagLemma}. 
\end{proof}

We note that the map of \autoref{lemJoin} is not an isomorphism for general PIDs; consider (for example) the pair of lines in $\Z^2$ from \autoref{ExampleIncompatibleLines}. This lemma is the one point in our paper we must assume that we are working over a field. 

We will combine  \autoref{lemJoin} with the celebrated Solomon--Tits theorem, to bound the connectivity of $T^{2,0}_n(F)$.  We note that the Solomon--Tits theorem holds for general PIDs $R$. In fact, when $R$ is a PID and $F=\mathrm{Frac}(R)$ is its field of fractions,  there are inverse isomorphisms
\begin{align*}  T_n(R) & \overset{\cong}{\longrightarrow} T_n(F) \\ 
 M & \longmapsto \mathrm{span}_F(M) \\ 
 V \cap R^n &  \longmapsfrom V 
\end{align*}

\begin{theorem}[Solomon--Tits Theorem] \label{SolomonTits} Let $R$ be a PID. Then $T_n(R)$ is $(n-3)$-connected. 
\end{theorem} 

 \autoref{lemJoin} and the Solomon--Tits Theorem (\autoref{SolomonTits})  imply the following. 

\begin{corollary}\label{CorT2Conn}
For $F$ a field, $T^{2,0}_n(F)$ is $(2n-4)$-connected.
\end{corollary}

\begin{proposition} \label{badness} 
Let $R$ be a PID such that  $T^{2,0}_n(R)$ is $(2n-4)$-connected. For $k \geq 2$ and $n \geq 3$, $T^{k,0}_n(R)$ and ${\CB}_n(R)$ are $1$-connected.  
\end{proposition}

By \autoref{CorT2Conn}, the proposition holds whenever $R$ is a field. 

\begin{proof}[Proof of \autoref{badness}]
By \autoref{colimTD}, it suffices to prove the claim for  $T^{k,0}_n(R)$.  Fix $k \geq 2$ and $n \geq 3$. By hypothesis, $T^{2,0}_n(R)$ has connectivity at least $(2n-4) \geq 2$.  Given any two vertices in $T^{k,0}_n(R)$, these two vertices are contained in a subspace isomorphic to $T^{2,0}_n(R)$, hence are connected by a path. Our goal is to show that $T^{k,0}_n(R)$ is simply connected. 

Consider a simplicial map $$g \colon  X \m T^{k,0}_n(R)$$ with $X$ some simplicial complex structure on $S^1$. If we can show that $g$ is homotopic to a map that factors through $T^{2,0}_n(R)$, then by our assumption on $T^{2,0}_n(R)$ we can conclude that $g$ is nullhomotopic and hence that every path component of $T^{k,0}_n(R)$ is $1$-connected. We will prove this  by showing that $g$ is homotopic to a map that factors through $T^{1,0}_n(R)$.

Call a simplex of $X$ \emph{bad} if  none of its vertices map to the image of $T^{1,0}_n(R)$. We are done if we can homotope $f$ to have no bad simplices. Suppose $\sigma=[x_0,x_1]$ is a bad edge. Let $\iota_j$ be the map induced by the injection from $\{1\}$ to $\N$ with image $j$, in the sense of \autoref{DefnInducedMap}. Let $V_0, V_1$ be submodules of $R^n$ and $n_0, n_1$ be numbers with $$g(x_0)=\iota_{n_0}(V_0) \qquad \text{and} \qquad g(x_1)=\iota_{n_1}(V_1) .$$ Since $\sigma$ is bad, $n_0$ and $n_1$ are larger than $1$. Observe that $g(\sigma)$ is contained in $\Link_{T^{k,0}_n(R)}(\iota_1(V_0))$. Let $X'$ be the subdivision of $X$ with a new vertex $t$ in the middle of $\sigma$. Let $g'\colon X' \m T^{k,0}_n(R)$ be the map that sends $t$ to $\iota_1(V_0)$ and that agrees with $g$ elsewhere. Since $[\iota_1(V_0),g(x_0),g(x_1)]$ is a simplex of $T^{k,0}_n(R)$, $g$ is homotopic to $g'$. By iterating this procedure, we can find a homotopic map $$g''\colon X'' \m T^{k,0}_n(R)$$ with $X''$ a new simplicial structure on $S^1$ and such that $g''$ has no bad $1$-simplices.

Finally, we will describe a procedure for removing isolated bad vertices. Let $y_1 \in X''$ be a bad vertex and let $g''(y_1)=\iota_j(V_1)$. Note that $j \geq 2$. Let $y_0$ and $y_2$ be the vertices adjacent to $y_1$ in $X''$. There are submodules $V_0$ and $V_2$ with  $g''(y_0)=\iota_1(V_0)$ and $g''(y_2)=\iota_1(V_2)$. Since $n \geq 3$, the complex $T_n(R)$ is connected. Thus there is a simplicial structure $Y$ on $[0,1]$ and a simplicial map $h\colon Y \m T_n(R)$ with $h(0)=V_0$ and $h(1)=V_2$. Let $X'''$ be the new simplicial structure on $S^1$ obtained from $X''$ by replacing $[y_0,y_1] \cup [y_1,y_2]$ with $Y$.  Define $g'''\colon X''' \m T^{k,0}_n(R)$  to agree with $g''$ on vertices of $X''$ excluding $y_1$,  and let $g''$ equal $\iota_1 \circ h$ on $Y$. Consider the loop obtained by concatenating the paths $g''\big|_{[y_0, y_1]\cup[y_1, y_2]}$ and $\iota_1 \circ h$. Its image is contained in a copy of $T^{2,0}_n(R)$, hence it is nullhomotopic by hypothesis. We infer that $g''$ and $g'''$ are homotopic.   We have removed one bad vertex. Iterating this procedure produces the desired homotopy. \end{proof}

\subsection{High connectivity of ${\CB}_n(F)$}

We first prove $D_n^{k,0}(R)$ is highly connected whenever $T^2_n(R)$ is.  The result for fields follows from \autoref{CorT2Conn}.

\begin{proposition} \label{DRognes}
Let $R$ be a PID such that $T^2_n(R)$ is $(2n-4)$-connected.  For $k \geq 2$, $D_n^{k,0}(R)$ is $(2n+k-3)$-connected.  In particular, when $F$ is a field, $D_n^{k,0}(F)$ is $(2n+k-3)$-connected for all $k \geq 2$. 
\end{proposition}

\begin{proof}
We will prove the claim by induction on $k$. By \autoref{suspend} and the assumption on $T^2_n(R)$,  $$D_n^{2,0}(F) \simeq \Sigma^3 (T^2_n(R))$$   is $(2n-1)$-connected. This proves the base case. Assume we have proven that $D_n^{k-1,0}(R)$ is $(2n+k-4)$-connected. 

We will apply \autoref{Bconn} to show $BD_n^{k-1,0}(R)$ is $(2n+k-3)$-connected. To do this, we must verify that $BD_n^{k-1,0}(R)_{\bullet}$ is a proper simplicial space. This holds because the latching objects are CW subcomplexes hence form a strong NDR-pair. By \autoref{Bsplit}, $BD_n^{k-1,0}(R) \cong D_n^{k-1,1}(R)$. By \autoref{SplitvsNotSplitD},  $D_n^{k-1,1}(R) \m D_n^{k,0}(R)$ is a homotopy equivalence and so $D_n^{k,0}(R)$ is $(2n+k-3)$-connected. The claim follows by induction.
\end{proof}

We now prove the following which includes the statement of \autoref{ConnectivityThm}. The statement for fields follows from \autoref{CorT2Conn}. 

\begin{theorem}\label{maingeneral} Let $R$ be a PID.  Let $k \geq 2$. If $T^2_n(R)$ is $(2n-4)$-connected, then
the complexes $T_n^k(R)$ and ${\CB}_n(R)$ are $(2n-4)$-connected.
In particular, when $F$ is a field, $T_n^k(F)$ and ${\CB}_n(F)$ are $(2n-4)$-connected for $k \geq 2$. 
\end{theorem}
\begin{proof}
Note that the claim is vacuous for $n=1$ and $n=0$. Since $D_n^{k,0}(R)$ is $(2n+k-3)$-connected by \autoref{DRognes}   and $\Sigma^{k+1} T_n^{k,0}(R) \simeq D_n^{k,0}(R)$ by \autoref{suspend},  $$\widetilde H_i(T_n^{k,0}(R)) \cong 0 \qquad \text{ for $i \leq 2n-4$.}$$ This implies $T_2^k(R)$ is $0$-connected since reduced homology detects $0$-connectivity. Now consider the case $n \geq 3$. By  \autoref{badness}  the space $T_n^{k,0}(R)$ is simply-connected for $n \geq 3$ and $k \geq 2$ under the assumption that $T^2_n(R)$ is $(2n-4)$-connected. Hence the Hurewicz theorem implies $T_n^k(R)$ is $(2n-4)$-connected. Since ${\CB}_n(R)$ is the homotopy colimit over $k$ of the spaces $T_n^k(R)$, we deduce that ${\CB}_n(R)$ is also $(2n-4)$-connected.

By \autoref{CorT2Conn}, the hypothesis that $T^2_n(F)$ is $(2n-4)$-connected holds whenever $F$ is a field. 
\end{proof}

\begin{remark}
Using Galatius--Kupers--Randal-Williams \cite[Theorem 7.1 (ii)]{e2cellsIV} instead of \autoref{CorT2Conn}, it is possible to prove a version of \autoref{maingeneral} for fields replaced with semi-local PIDs with infinite residue fields (e.g. power series rings of infinite fields in one variable).

\end{remark}

\section{The Koszul dual of the Steinberg monoid} \label{SectionKoszulDual}

We now review the definition of the Steinberg monoid and compute its Koszul dual.

\begin{definition}
For $R$ a PID, let $\St(R)$ be the $\VB(R)$-ring with $\St(R)(M)=\widetilde H_{\rank(M)}(D^{1,0}(M))$ and with ring structure induced by the monoid structure on $D^{1,0}$. 
\end{definition} 

The ring structure on $\St(R)$ was originally introduced by Miller--Nagpal--Patzt \cite[Section 2.3]{MNP}. The definition given here is due to Galatius--Kupers--Randal-Williams \cite[Lemma 6.6]{e2cellsIV} and agrees with that of Miller--Nagpal--Patzt  by Galatius--Kupers--Randal-Williams \cite[Remark 6.7]{e2cellsIV}.

The monoid $\St(R)$ is augmented and connected in the sense of \autoref{DefnAugmented} since $\St(R)_0 \cong \Z$. Recall from \autoref{DefnKoszul} that, given an augmented $\VB$-ring $A$, we say $A$ is \emph{Koszul} if $\Tor_i^{A}(\Z,\Z)_n \cong 0$ for $i \neq n$, and its Koszul dual is the  $\VB$-module $$M \longmapsto \Tor_{\rank M}^{A}(\Z,\Z)_M.$$
It was shown by Miller--Nagpal--Patzt \cite[Theorem 1.4]{MNP} that $\St(F)$ is Koszul if $F$ is a field. We give a new proof of this and compute its Koszul dual. To do this, we first give a description of $\Tor_i^{\St(R)}(\Z,\Z)_n$ for any PID using the Solomon--Tits theorem. 
\begin{lemma} \label{TorD}
For $R$ a PID, there is a natural isomorphism $\Tor_i^{\St(R)}(\Z,\Z)_n \cong \widetilde H_{i+n}(D^{2,0}_n(R))$.    
\end{lemma}

\begin{proof}
Recall from \autoref{Bsplit} and \autoref{SplitvsNotSplitD}, $$B D^{1,0}_n(R) \cong D^{1,1}_n(R) \simeq D^{2,0}_n(R).$$  Since $D^{1,0}_n(R) \simeq \Sigma^2 T^{1,0}_n(R)$ by \autoref{suspend}, the Solomon--Tits theorem (\autoref{SolomonTits}) implies that $\widetilde H_i ( D^{1,0}(R) ) =0$ for $i<n$ and by definition $\widetilde H_i ( D^{1,0}_n(R) ) = \St(R)_n$ for $i=n$. 

Consider the hypertor spectral sequence where $(E^2_{p,q})_n$ is the homological degree-$q$ part of $\Tor_p^{\widetilde{H}_*(D^{1,0}_n(R))}(\Z, \Z) $. The $d_r$ differential is a map $$d_r \colon \left(E^r_{p,q}\right)_n \to \left(E^r_{p-r, q-1+r}\right)_n.$$  The spectral sequence converges to $\Tor_{p+q}^{\widetilde{C}_*(D^{1,0}_n(R))}(\Z, \Z)$.  Since the homology of $\widetilde H_i ( D^{1,0}(R) )$ is concentrated in degree $i=n$ (where it is $\St(R)_n$), this spectral sequence collapses.  We see that $$\Tor_i^{\St(R)}(\Z,\Z)_n \cong \Tor_{n+i}^{\widetilde C_*(D^{1,0}_n(R)) }(\Z,\Z)_n.$$ It now suffices to show that  
$$\Tor_{j}^{\widetilde C_*( A) }(\Z,\Z)_n \cong {H}_j( B(\widetilde{C}_*(A)) ) \cong \widetilde H_j(B A)$$ 
for $A$ a connected augmented monoid object in simplicial $\VB$-based spaces. 
The first isomorphism is a chain-level version of \autoref{PropositionTorViaBarConstruction}. The second isomorphism is a version of the Eilenberg--Zilber theorem. 
\end{proof}

The following theorem implies the Koszulness result of Miller--Nagpal--Patzt \cite[Theorem 1.4]{MNP}  and our \autoref{KD}, which states that for $F$ a field, $\Tor_n^{\St(F)}(\Z,\Z)_n \cong \St_n(F) \otimes \St_n(F)$.

\begin{theorem} Let $R$ be a PID. Let $n, i \geq 0$. There are isomorphisms of $\GL_n(R)$-representations
$$\Tor_i^{\St(R)}(\Z,\Z)_n \cong \widetilde{H}_{i+n-3}(T^{2}(R))$$ 
Thus, if $R$ is a PID such that $T^{2}(R)$ is $(2n-4)$-connected, then $\St(R)$ is Koszul. 
In particular, when $F$ is a field,  then $\St(F)$ is Koszul and there are isomorphisms of $\GL_n(F)$-representations
$$\Tor_i^{\St(F)}(\Z,\Z)_n \cong \left\{ \begin{array}{ll}  \St_n(F) \otimes \St_n(F), &  i=n \\ 0, & i \neq n. \end{array} \right. $$ 
\end{theorem} 

\begin{proof}
By \autoref{TorD}, $$\Tor_i^{\St(R)}(\Z,\Z)_n \cong \widetilde H_{i+n}(D^{2,0}_n(R)).$$ By \autoref{suspend}, $$\widetilde H_{i+n}(D^{2,0}_n(R)) \cong \widetilde H_{i+n-3}(T^{2,0}(R)).$$  By \autoref{lemJoin}, when $F$ is a field,  $$\widetilde H_{i+n-3}(T^{2,0}(F))  \cong \left\{ \begin{array}{ll}  \St_n(F) \otimes \St_n(F), &  i=n \\ 0, & i \neq n. \end{array} \right.$$ 
\end{proof}


\bibliographystyle{amsalpha}
\bibliography{refs}

\vspace{.5cm}

\end{document}